\documentclass[11pt, final]{amsart}
\usepackage[english]{babel}
\usepackage{latexsym}
\usepackage{amssymb}
\usepackage{amscd}
\usepackage[all]{xy}
\usepackage[cp850]{inputenc}
\usepackage[mathscr]{eucal}
\theoremstyle{plain}
\newtheorem{prob}{Problem}

\newtheorem{teor}{Theorem}[section]%[chapter]
\newtheorem{prop}[teor]{Proposition}%[chapter]
\newtheorem{cor}[teor]{Corollary}%[chapter]

\newcommand{\adef}{\begin{defn}}
\newcommand{\zdef}{\end{defn}}
 \newtheorem{defn}[teor]{Definition}
\newcommand{\cl}{\mathcal}
\newtheorem{lemma}[teor]{Lemma}%[chapter]

\theoremstyle{definition}

\newcommand{\lop}{\ensuremath{\curvearrowright}}
\theoremstyle{remark}

\newcommand{\sn}{\sum_{k=1}^n}

\newcommand{\R}{\mathbb{R}}
\newcommand{\C}{\mathbb{C}}
\newcommand{\N}{\mathbb{N}}

\newcommand{\seq}{\ensuremath{0\to Y\to X\to Z\to 0\:}}

\newcommand{\aproof}{\begin{proof}}
\newcommand{\zproof}{\end{proof}}

\newcommand\PO{{\mathrm{PO}}}

\newcommand{\To}{\longrightarrow}

\author{Jes\'us M. F. Castillo}
\address{Departamento de Matem\'aticas, Universidad de
Extremadura, Avenida de Elvas s/n, 06071 Badajoz, Espa\~{n}a} \email{castillo@unex.es}

\author{ Wilson Cuellar} \address{Instituto de Matem\'atica e Estadistica, Universidade de S\~{a}o Paulo, R. do Mat\~{a}o, 1010 - Butanta, S\~{a}o Paulo, Brazil} \email{cuellar@ime.usp.br}

\author{Valentin Ferenczi} \address{Departamento de Matem\'atica, Instituto de Matem\'atica e
Estat\'\i stica, Universidade de S\~ao Paulo, rua do Mat\~ao 1010,
05508-090 S\~ao Paulo SP, Brazil, and Equipe d'Analyse Fonctionnelle, Institut de Math\'ematiques de Jussieu,
Universit\'e Pierre et Marie Curie - Paris 6, Case 247, 4 place Jussieu,
75252 Paris Cedex 05, France} \email{ferenczi@ime.usp.br}

\author{Yolanda Moreno} \address{Departamento de Matem\'aticas, Escuela Polit\'ecnica, Universidad de Extremadura,
Avenida de la Universidad s/n, 10071 C\'aceres, Espa\~{n}a} \email{ymoreno@unex.es}

\thanks{The research of the first and fourth authors has been supported in part by project MTM2013-45643-C2-1-P, Spain. The second and third author were supported by Fapesp, grants 2010/17512-6, 2013/11390-4, 2014/25900-7}

\begin{document}

\title{Complex structures on twisted Hilbert spaces}

\maketitle

\begin{abstract} We investigate complex structures on twisted Hilbert spaces, with special attention paid to the Kalton-Peck $Z_2$ space and to the hyperplane problem. We consider (nontrivial) twisted Hilbert spaces generated by centralizers obtained from an interpolation scale of K\"othe function spaces. We show there are always complex structures on the Hilbert space that cannot be extended to the twisted Hilbert space. If, however, the scale is formed by rearrangement invariant K\"othe function spaces then there are complex structures on it that can be extended to a complex structure of the twisted Hilbert space.  Regarding the hyperplane problem we show that no complex structure on $\ell_2$ can be extended to a complex structure on an hyperplane of $Z_2$ containing it.\end{abstract}

\section{Introduction}

A twisted Hilbert space is a Banach space $X$ admitting a subspace $H_1$ isomorphic to a Hilbert space and so that the quotient $X/H_1$ is also isomorphic to a Hilbert space $H_2$. Or else, using the homological language, a Banach space $X$ that admits an exact sequence

$$\begin{CD} 0@>>> H_1 @>>> X @>>> H_2 @>>> 0\end{CD}$$in which both $H_1$ and $H_2$ are Hilbert spaces. The space $X$ is usually denoted $H_1 \oplus_\Omega H_2$ and the so-called quasi-linear map $\Omega$ is there to specify the form in which the norm of the direct product must be ``twisted". The most interesting example for us is the Kalton-Peck space $Z_2$ \cite{kaltpeck}, which is  the twisted Hilbert space associated to the quasi-linear map (defined on finitely supported sequences as)
  $$ \mathscr K \left (\sum_i x_ie_i \right ) = \sum_i x_i \left ( \log \frac{|x_i|}{\|x\|_2}\right )e_i $$
with the understanding that $\log(|x_i|/ \|x\|_2)=0$ if $x_i=0$.\medskip

This paper deals with different aspects of the existence of complex structures on twisted Hilbert spaces and in particular on $Z_2$. In the first place, one finds the\medskip

\textbf{Existence problem:}  Does every twisted Hilbert space admit a complex structure?\medskip

But of course one can be more specific. \medskip

\textbf{Extension problem:}  Given a twisted Hilbert space $H_1 \oplus_\Omega H_2$, can every complex structure defined on $H_1$ (resp. $H_2$) be extended (resp. lifted) to a complex structure on $H_1 \oplus_\Omega H_2$?\medskip

Our motivation for the study of the extension problem is the hyperplane problem, which has its origins in Banach's book \cite{B}, and asks whether  every closed hyperplane (i.e., 1-codimensional closed subspace) of any infinite dimensional Banach space is linearly isomorphic to the whole space. The first space conjectured not to be isomorphic to its hyperplanes was precisely the space $Z_2$, and whether it is so is still an open problem. Banach's hyperplane problem was solved by W.T. Gowers in 1994, \cite{G}, with  the  construction of a Banach space with unconditional basis  not isomorphic to any of its proper subspaces. His space is an unconditional version of the first hereditarily indecomposable (in short H.I.) space constructed by Gowers and Maurey \cite{GM}; after Gowers' result,   Gowers and Maurey showed  that their space --as, in fact, any H.I. space-- cannot be isomorphic to any of its proper subspaces. Since then many other examples "with few operators" and consequently not isomorphic to their hyperplanes, have appeared, including Argyros-Haydon's space on which every operator is a compact perturbation of a multiple of the identity \cite{AH}. In comparison to those examples the Kalton-Peck space has a relatively simple and natural construction;  as will be recalled later, it also appears as naturally induced by the interpolation scale of $\ell_p$ spaces. This motivates our study of the hyperplane problem in this setting. Our quest for an isomorphic invariant to distinguish between $Z_2$ and its hyperplanes lead us to consider whether the hyperplanes of $Z_2$ admit complex structures; or else, in terms of Ferenczi-Galego  \cite{feregale}, we investigate whether $Z_2$ is  "even".\medskip

The answers we present for the existence and extension problems are the following. For twisted Hilbert spaces generated by centralizers obtained from an interpolation scale of rearrangement invariant K\"othe spaces, we observe a positive answer to the existence problem. For non-trivial twisted Hilbert spaces generated by centralizers obtained from an interpolation scale of K\"othe spaces we provide a negative answer to the extension problem. Regarding the hyperplane problem we show that no complex structure on $\ell_2$ can be extended to a complex structure on an hyperplane of $Z_2$ containing it; more generally, that hyperplanes of $Z_2$ do not admit, in our terminology, compatible complex structures.

\section{Complex structures on interpolation scales and twisted sums}

\subsection{Complex structures} A complex structure on a real Banach space $X$  is a complex Banach space $Z$ which is $\R$-linearly isomorphic to $X$.   Let $Z$ be a complex structure on $X$ and $T:X \to Z$ be  a $\R$-linear isomorphism; then $X$ can be seen as a complex space where  the multiplication by the complex number $i$ on vectors of $X$ is given by the operator $ \tau x= T^{-1}(iTx)$ which is clearly  an automorphism on $X$ satisfying $\tau^2=-id_X$.   Conversely, if there exists a linear operator  $\tau$ on $X$ such that $\tau^2=-id_X$, we can define on $X$    a  $\mathbb C$-linear structure by declaring a new law for the scalar multiplication: if  $\lambda, \mu \in \mathbb R$ and $x\in X$, then $(\lambda + i\mu).x= \lambda x + \mu \tau(x)$. The resulting complex Banach space will be denoted by $X^\tau$.

\adef A real Banach space $X$ admits a complex structure if there is a bounded linear operator $
u: X\to X$ such that $u^2 = - id_X$. Two complex structures $u, v$ on $X$ are equivalent if there is a linear automorphism $\phi$ of $X$ such that  $u = \phi v \phi^{-1}$ or, equivalently, when $X^u$ and $X^v$ are $\C$-linearly isomorphic.
\zdef

%The space $X^I$  is the complex structure of $X$ associated to the operator $I$. We will often refer  to the   operator  $I$  itself as a complex structure on $X$.

 The first example in the literature of a real Banach space that does not admit a complex structure was the  James space, proved by Dieudonn\'e \cite{D}. Other  examples of spaces without complex structures are the uniformly convex space constructed by Szarek \cite{S2} or the hereditary indecomposable space of  Gowers and Maurey  \cite{GM}.
 At the other extreme, we find the classical spaces $\ell_p$ or $L_p$, which admit a unique complex structure up to $\C$-linear isomorphism. A simple proof of this fact was provided by N.J. Kalton and appears in  \cite[Thm. 22]{FG}. In the case of $\ell_p$ spaces we will call $\omega$ the following complex structure(s):
 let $\sigma: \N\to \N$ be the permutation $\sigma = (2,1)(4,3) \dots (2n, 2n-1) \dots $ and then set
$$\omega((x_n)) = ((-1)^n x_{\sigma(n)}).$$

In addition to that, there are currently known examples of spaces admitting  exactly $n$ complex structures (Ferenczi \cite{F}), infinite countably many (Cuellar \cite{C}) and uncountably many (Anisca \cite{A}).

 Most of the action in this paper will take place in the ambient of K\"othe  functions spaces over a $\sigma$-finite measure space $(\Sigma, \mu)$ endowed with their $L_\infty$-module structure; a particular case of which is that of Banach spaces with a 1-unconditional basis (called Banach sequence spaces in what follows) with their associated $\ell_\infty$-structure. We will denote by $L_0$ (resp. $L_p$) the space of all $\mu$-measurable (resp. measurable and $p$-integrable) functions, without further explicit reference to the underlying measure space. In the case of $L_p$ spaces, a complex structure can be obtained by fixing a partition $(\Sigma_n)_n$ of $\Sigma$ on positive measure sets for which there is measure preserving bijection $a_n$ between $\Sigma_{n}$ and $\Sigma_{\sigma(n)}$ for all $n$; then for $f= \sum f_{|\Sigma_n} 1_{\Sigma_n} $, set
 \begin{equation}\label{rispace}\omega(f) =\sum (-1)^n ( f_{|\Sigma_{\sigma(n)}}\circ a_n) 1_{\Sigma_n}.\end{equation}

\subsection{Exact sequences, quasi-linear maps and centralizers}

For a rather complete background on the theory of twisted sums see \cite{castgonz}. We recall that  a twisted sum of two Banach spaces $Y$,  $Z$ is  a quasi-Banach space $X$ which has a closed subspace isomorphic to $Y$ such that the quotient $X/Y$ is isomorphic to $Z$. Equivalently, $X$ is a twisted sum of $Y$, $Z$ if there exists a short exact sequence
$$\begin{CD}
0@>>>  Y@>>> X @>>> Z@>>> 0.\end{CD}$$

According to Kalton  and Peck \cite{kaltpeck}, twisted sums  can be identified with homogeneous maps $\Omega: Z \to Y$ satisfying
\[ \| \Omega (z_1+z_2) - \Omega z_1- \Omega z_2\| \leq C(\|z_1\|+ \|z_2\|),\]
which are called  quasi-linear maps, and induce an equivalent quasi-norm on $X$ (seen
algebraically as $Y \times Z$) by $$\|(y,z)\|_\Omega=\|y-\Omega z\| +\|z\|.$$
%This is usually extended to maps with values in some linear space $Y_0$ containing $Y$ and for which $\Omega(z_1+z_2)-\Omega z_1 -\Omega z_2 \in Y$ and the above inequality still holds; to include this situation we  write $\Omega: Z \lop Y$.

When $Y$ and $Z$ are, for example,  Banach spaces  of non-trivial type, the quasi-norm above is equivalent to a norm;  therefore, the twisted sum obtained is a Banach space. And when $Y$ and $Z$ are both isomorphic to $\ell_2$, the twisted sum is called a {\em  twisted Hilbert space}. Two exact sequences $0 \to Y \to X_1 \to Z \to 0$ and $0 \to Y \to X_2 \to Z \to 0$
are said to be {\it equivalent} if there exists an operator $T:X_1\to X_2$ such that the following
diagram commutes:
$$
\begin{CD}
0 @>>>Y@>>>X_1@>>>Z@>>>0\\
&&@| @VVTV @|\\
0 @>>>Y@>>>X_2@>>>Z@>>>0.
\end{CD}$$
The classical 3-lemma (see \cite[p. 3]{castgonz}) shows that $T$ must be an isomorphism.
An exact sequence is trivial if and only if it is equivalent to
$0 \to Y \to Y \oplus Z \to Z \to 0$, where $Y\oplus Z$ is endowed with the product norm.
In this case we say that the exact sequence \emph{splits.} Two quasi-linear maps $\Omega, \Omega': Z \to Y$ are said to be equivalent, denoted $\Omega\equiv \Omega'$,
if the difference $\Omega-\Omega'$ can be written as $B +L$, where $B: Z \to Y$ is a homogeneous bounded
map (not necessarily linear) and $L: Z \to Y$ is a linear map (not necessarily bounded).
Two quasi-linear maps are equivalent if and only if the associated exact sequences
are equivalent. A quasi-linear map is trivial if it is equivalent to the $0$ map,  which also means that  the associated exact sequence is trivial. Given two Banach spaces $Z,Y$ we will denote by $\ell(Z,Y)$ the vector space of linear (not necessarily continuous) maps $Z\to Y$. The distance between two homogeneous maps $T,S$ will be the usual operator norm (the supremum on the unit ball) of the difference; i.e., $\|T-S\|$, which can make sense even when $S$ and $T$ are unbounded. So a quasi-linear map $\Omega: Z \To Y$ is trivial if and only if $d(\Omega,\ell(Z,Y))\leq C <+\infty$, in which case we will say that $\Omega$ is $C$-trivial.\medskip

Given an exact sequence \seq with associated quasi-linear map $\Omega$ and an operator
$\alpha: Y\to Y'$, there is a commutative diagram
\begin{equation}\label{po}
\begin{CD}
0 @>>>Y@>>>Y\oplus_\Omega Z @>>>Z@>>>0 \\
&&@V{\alpha}VV @VV\alpha'V @|\\
0 @>>>Y'@>>> \PO @>>>Z@>>>0\end{CD}
\end{equation}
whose lower sequence  is called the \emph{push-out sequence} and the space $\PO$ is called the \emph{push-out space}: just set $\PO = Y' \oplus_{\alpha \Omega} Z$
and $\alpha'(y, z)  = (\alpha y, z)$.

\subsection{Compatible complex structures on twisted sums} The approach to complex structures as operators allows us to consider their extension/lifting properties. The commutator notation will be helpful here. Recall that given two maps $A,B$ for which what follows makes sense, their commutator is defined as $[A,B] = AB -BA$. Given three maps $A, \Omega, B$ for which what follows makes sense, we define their ``commutator" as $[A,\Omega, B] = A\Omega  - \Omega B$.

\adef  Let $Z,Y$ be Banach spaces and $\Omega: Z \to Y$ be  a quasi-linear map. Let  $\tau $ and $\sigma $ be two operators on $Y$ and $Z$, respectively. The couple $(\tau, \sigma)$ is compatible with $\Omega$ if $\tau\Omega \equiv\Omega \sigma$.
\zdef

The following are well-known equivalent formulations to compatibility.

\begin{lemma}\label{crit} The following are equivalent:
\begin{itemize}
\item $(\tau, \sigma)$ is compatible with $\Omega$.
\item $[\tau, \Omega, \sigma]\equiv 0$.
\item $\tau$ can be extended to an operator $\beta: Y\oplus_\Omega Z \to Y\oplus_\Omega Z$ whose induced operator on the quotient space is $\sigma$.
\item $\sigma$ can be lifted to an operator $\beta: Y\oplus_\Omega Z \to Y\oplus_\Omega Z$ whose restriction to the subspace is $\tau$.
    \item There is a commutative diagram
\begin{equation}\label{compa}\begin{CD}
0@>>> Y @>>>  Y \oplus_\Omega Z @>>>  Z @>>> 0 \\
&&@V{\tau}VV @VV{\beta}V @VV{\sigma}V\\
0@>>> Y @>>> Y \oplus_{\Omega }Z @>>>  Z @>>> 0;
\end{CD}\end{equation}
\end{itemize}
\end{lemma}

\begin{proof} That $[\tau, \Omega, \sigma] \equiv 0$ means that $\tau \Omega - \Omega \sigma = B + L$ for some homogeneous bounded map $B$ and some linear map $L$. In which case, the operator $\beta (y,z) = (\tau y + Lz, \sigma z)$ is linear and continuous, and makes the diagram (\ref{compa}) commutative.\end{proof}

The last commutativity condition can be formulated without an explicit reference to any associated quasi-linear map. So, $(\tau, \sigma)$ make a commutative diagram
$$\begin{CD}
0@>>> Y @>>>  X @>>>  Z @>>> 0 \\
&&@V{\tau}VV @VV{\beta} V @VV{\sigma}V\\
0@>>> Y @>>> X @>>>  Z @>>> 0
\end{CD}$$
if and only $(\tau, \sigma)$ is compatible with any quasi-linear map $\Omega$ associated to the exact sequence.\medskip

A simple computation shows that if,  moreover, $\tau, \sigma$ are complex structures then the operator $\beta$ defined above is a complex structure, $\beta^2=-id$, if and only if
$\tau L + L \sigma =0$. Note also that when $(\tau, \sigma)$ is compatible with $\Omega$,
the decomposition $\tau \Omega - \Omega \sigma = B+L$ above is not unique; which makes worthwhile the following corollary.

\begin{cor}\label{bcommu} Let $\tau$ and $\sigma$ be complex structures defined on Banach spaces $Y$ and $Z$, respectively.  Let $\Omega: Z \to Y$ be  a quasi-linear map. If $[\tau,\Omega, \sigma]$ is linear or bounded then $\tau$ can be extended to a complex structure $\beta: Y\oplus_\Omega Z \to Y\oplus_\Omega Z$ whose induced operator on the quotient space is $\sigma$. \end{cor}
\begin{proof} If  $[\tau,\Omega, \sigma]$ is linear, set  $L = [\tau,\Omega, \sigma]$ and $B=0$.
This makes $\tau L + L \sigma = \tau(\tau\Omega - \Omega \sigma) + (\tau\Omega- \Omega \sigma)\sigma = -\Omega  - \tau\Omega \sigma + \tau\Omega \sigma + \Omega = 0$, and apply the observation after Lemma \ref{crit} above. If $[\tau, \Omega, \sigma]$ is bounded then set $B = [\tau, \Omega, \sigma]$ and $L=0$. \end{proof}

Observe that in the case in which $[\tau, \Omega, \sigma]$ is bounded, the associated complex structure $\beta$ on $Y\oplus_\Omega Z$ appearing in the corollary is $\beta(y, z) = (\tau y, \sigma z)$ and thus it will be denoted as $\beta=(\tau, \sigma)$. Such is the natural situation that occurs when working with centralizers, a crucial notion we shall now recall.

\adef An $L_\infty$-centralizer on a K\"othe function space $X$ (resp. an $\ell_\infty$-centralizer on a Banach sequence space) is a quasi-linear map $\Omega: X \to L_0$  such that  there is a constant $C$  satisfying that,  for every $f\in L_\infty$ (resp. $\ell_\infty$) and for every $x\in X$, the difference $\Omega(fx)- f\Omega(x)$ belongs to $X$ and
$$ \| \Omega(fx)- f\Omega(x)\|_X\leq C\|f\|_\infty  \|x\|_X. $$
\zdef
Observe that a centralizer on $X$ does not take values in $X$, but in $L_0$. It is however true that for all $x,y\in X$ one has $\Omega(x+y) - \Omega(x) - \Omega(y) \in X$. For this reason it induces an exact sequence
$$\begin{CD}
0@>>>  X@>j>> d_\Omega X @>q>> X@>>> 0\end{CD}$$
as follows: $d_\Omega X= \{ (w, x) : w\in L_0, x\in X: w - \Omega x\in X\}$, with
quasi-norm $$\|(w,x)\|_{d_\Omega X}=\|x\|_X+\|w- \Omega x\|_X$$ and with obvious inclusion $j(x) = (x, 0)$ and quotient map $q(w,x)=x$. We will use the notation $\Omega: X \lop X$ instead of $\Omega: X \To L_0$ to describe that fact. The induced  sequence is trivial if and only if there is a linear map $L: X\to L_0$ so that $\Omega - L: X\to X$ is bounded. Centralizers arise naturally induced by interpolation, as we describe next.

\subsection{Complex interpolation and twisted sums}\label{complex} The complex interpolation method we follow is that of  \cite{Bergh-Lofstrom}. A couple of complex Banach spaces $(X_0,X_1)$ is said to be  compatible  if   their intersection $\Delta$ is dense in each of the spaces and both  spaces can be continuously  embedded  in a topological vector  space $V$.  We can form their sum $\Sigma= X_0+X_1 \subset V$ endowed with the norm $$\| x\|_{\Sigma}= \inf  \|x_0\|_0 + \|x_1\|_1,$$ where the infimum  is taken  over all representations of $x=x_0+x_1$ with $x_i\in X_i$ (i=0,1). Let $\mathbb S$ denote the closed strip $\{z\in \mathbb C: 0\leq \Re z\leq 1\}$ in the complex plane,
and let $\mathbb S^\circ$ be its interior.  We denote by $\cl{H}=\cl H(X_0,X_1)$  the space of functions $g: \mathbb S \to \Sigma$  satisfying the following conditions:

\begin{enumerate}
\item $g$ is $\|\cdot\|_\Sigma$-bounded and $\|\cdot\|_\Sigma$-continuous on $\mathbb S$,
and $\|\cdot\|_\Sigma$-analytic on $\mathbb S^\circ$;

\item $g(it)\in X_0$ for each $t\in\R$, and the map $t\in\R\mapsto g(it)\in X_0$ is bounded
and continuous;

\item $g(it+1)\in X_1$ for each $t\in\R$, and the map $t\in\R\mapsto g(it+1)\in X_1$ is bounded
and continuous;
\end{enumerate}
The space $\cl{H}$ is a Banach space under the norm
$\|g\|_\cl{H} = \sup\{\|g(j+it)\|_j: j=0,1; t\in\R \}$.
For $\theta\in[0,1]$, define the interpolation space
$$
X_\theta=(X_0,X_1)_\theta=\{x\in\Sigma: x=g(\theta) \text{ for some } g\in\cl H\}
$$
with the norm $\|x\|_\theta=\inf\{\|g\|_H: x=g(\theta)\}$.
Let $\delta_{\theta}$ denote  the evaluation map $\delta_{\theta}: \cl H \to \Sigma$ given by $\delta_{\theta}(g)=g(\theta)$, which is bounded by the definition of $\cl H$.
So $X_\theta$ is the quotient of $\cl{H}$ by $\ker\delta_\theta$,
and thus it is a Banach space.\medskip

We will focus on the Schechter's  version of the complex method of interpolation which uses as interpolators  the maps $\delta_{\theta}$ and  the evaluation of  the derivative at $\theta$, $\delta_{\theta}' : \cl H \to \Sigma$  given by $\delta_{\theta}'  (g)=g'(\theta)$.
The key of the   Schechter method is  the following  result  (see \cite[Theorem 4.1]{carro}):

\begin{lemma}\label{mechanism}
$\delta'_\theta: \ker \delta_\theta\to X_\theta$ is bounded and onto for $0<\theta<1$.
\end{lemma}

%The higher order versions of this result, which will be needed in
%Section \ref{sect:higher-twisting}, also hold (see \cite{carrohigh}) and will be
%mentioned as required.

Lemma \ref{mechanism} provides the connection with exact sequences and twisted sums
through the following push-out diagram:
\begin{equation}\label{basic}\begin{CD}
0@>>> \ker\delta_\theta @>i_\theta>> \mathcal H @>\delta_\theta >> X_\theta @>>>0\\
&&@V\delta_\theta' VV @VVV @| \\
0@>>> X_\theta @>>> \PO @>>> X_\theta@>>> 0
\end{CD}
\end{equation}
whose lower row is obviously a twisted sum of $X_\theta$. Thus, if  $B_\theta: X_\theta \to \mathcal H$ is  a bounded
homogeneous selection for $\delta_\theta$, and  $L_\theta: X_\theta \to \mathcal H$ is  a linear selection, the quasi-linear  map naturally associated to the push-out sequence is $\delta_\theta'(B_\theta- L_\theta): X_\theta \to \Sigma$. Hence, $\PO \simeq X_\theta \oplus_{\delta_\theta'(B_\theta- L_\theta)} X_\theta.$  Clearly, the  quasi-linear map  depends on the choice  of $B_\theta$ and $L_\theta$.  However, if $\delta_\theta'(\tilde B_\theta- \tilde L_\theta)$ is another associated quasi-linear map then their difference is the sum of a bounded plus a linear map. This is what makes more interesting to work with the centralizer $\delta_\theta' B_\theta: X_\theta \to \Sigma$ and the induced exact sequence
$$\begin{CD} 0@>>> X_\theta @>>> d_{\delta_\theta' B_\theta}X_\theta @>>> X_\theta@>>> 0,
\end{CD}$$
(which is equivalent to the one induced by the quasi-linear map $\delta_\theta'(B_\theta- L_\theta)$ since the diagram

\begin{equation}\begin{CD}
0@>>> X_\theta @>>> X_\theta \oplus_{\delta_\theta'(B_\theta- L_\theta)} X_\theta @>>> X_\theta @>>>0\\
&&@| @VTVV  @| \\
0@>>> X_\theta @>>>  d_{\delta_\theta'B_\theta}X_\theta @>>> X_\theta@>>> 0
\end{CD}
\end{equation}

is commutative setting $T(x,z) = (x - \delta_\theta'L_\theta z, z)$). Indeed, for any other choice $\delta_\theta' \tilde B_\theta$ the difference $\delta_\theta' \tilde B_\theta - \delta_\theta' B_\theta$ is bounded:

\begin{equation}\label{select}
\|\delta'_\theta(\tilde B_\theta - B_\theta)x\|_{X_\theta}\leq
\|{\delta_\theta'}_{|\ker \delta_\theta}\| (\|\tilde B_\theta\| +\|B_\theta\|)\|x\|_{X_\theta}.
\end{equation}

We will from now on denote the centralizer $\delta_\theta' B_\theta$ as $\Omega_\theta$ independently of the choice of $B_\theta$. Thus, we are ready to prove a result similar to the bounded case of Corollary \ref{bcommu}, for which we only need one more definition.  An operator $T: X_{\theta} \to X_{\theta}$  is said to be an operator \emph{on the scale} if $T: \Delta \to \Delta$ and  $T: \Sigma \to \Sigma$ are both continuous; which, by classical results of interpolation theory,  necessarily makes $T: X_\theta \to X_\theta$ continuous too. Note that if $w-\Omega_\theta x \in X_\theta$ then since $\Omega_\theta x$ belongs to $\Sigma$, $w$ must belong to $\Sigma$. This means that $d_{\Omega_\theta}$ may be represented as
 $$d_{\Omega_\theta} X_\theta= \{ (w, x) : w\in \Sigma, x\in X_\theta: w - \Omega x\in X_\theta\}.$$ In this setting, for $\tau$ (resp. $\sigma$) defined on $\Sigma$ (resp. $X_\theta$),
$(\tau, \sigma): d_{\Omega_\theta} X_\theta \To \Sigma \times X_\theta$ is defined by $(\tau, \sigma) (w,x)=(\tau(w), \sigma(x))$. One then has:

\begin{prop} \label{comint} Let $(X_{\theta})$ be an interpolation scale. If $u: X_{\theta} \to X_{\theta}$ is a complex structure on the scale then
$$(u,u): d_{\Omega_{\theta}} X_{\theta}  \To d_{\Omega_{\theta}} X_{\theta} $$  is a  complex structure.
\end{prop}
\begin{proof} It is enough to check that $(u,u)$ is bounded on $d_{\Omega_{\theta}}$. Since the diagram
\begin{equation}\label{basic}\begin{CD}
0@>>> X_\theta @>>>  d_{\Omega_{\theta}} X_{\theta}  @>>> X_\theta @>>>0\\
&&@VuVV @V{(u,u)}VV @VVuV\\
0@>>> X_\theta @>>>  d_{\Omega_{\theta}} X_{\theta}  @>>> X_\theta@>>> 0
\end{CD}
\end{equation}
is commutative, it remains to show that $(u,u)$ is bounded. Which is equivalent to saying that $[u, \Omega_\theta]$ is bounded. That this holds is the content of the commutator theorem of  Coifman-Rochberg-Weiss  \cite{rochberg-weiss}; which can be verified  by the next inequality:
$$
\|u \delta'_\theta B_\theta x - \delta_\theta' B_\theta u x\|_{X_\theta} = \|\delta'_\theta u  B_\theta x - \delta_\theta' B_\theta ux\|_{X_\theta} \leq 2
\|{\delta_\theta'}_{|\ker \delta_\theta}\| \|u\| \|B_\theta\| \|x\|_{X_\theta}.
$$
\end{proof}

In other words, when the extremal spaces of the scale admit the same complex structure then so do all intermediate spaces of the scale as well as the derived spaces. Actually, what occurs is that when $u$ is a complex structure on the scale then also Calderon's function space $\mathcal H$ admits a complex structure $U: \mathcal H \to \mathcal H$ given by $U(f)(\theta) = u(f(\theta))$. On the other hand, derived spaces are a particular instance of Rochberg's derived spaces \cite{rochberg}. Indeed, Rochberg's description of the derived space is:
$$d_{\Omega_\theta}X_\theta = \{\big(f'(\theta), f(\theta)\big)\,
:\, f\in  \mathcal{H} \}$$
endowed with the obvious quotient norm of $\mathcal H/(\ker \delta_\theta \cap \ker \delta'_\theta)$
(see \cite[Prop. 3.2]{cfg} for an explicit proof). In this line one can also consider Rochberg's iterated spaces $\mathcal Z^{n+1} = \mathcal H/ \bigcap_{i=0}^{i=n} \ker \delta_\theta^i
$ and show \cite{cck} that they form exact sequences
$$\begin{CD} 0@>>> \mathcal Z^{n}@>>> \mathcal Z^{n+m}@>>> \mathcal Z^{m}@>>> 0.\end{CD}$$
Since obviously $\mathcal Z^2 = d_{\Omega_\theta}X_\theta$ we get that these spaces provide iterated twisted sums of the previously obtained twisted sum spaces (again, see \cite{cck} for different detailed descriptions). One has:

\begin{prop}\label{comintn} Let $(X_{\theta})$ be an interpolation scale. If  $u: X_{\theta} \to X_{\theta}$ is a complex structure on the scale then, for all $n\geq 1$, the spaces $\mathcal Z^{n+1}$ admit a  complex structure. More precisely,
$$(u, \stackrel{n+1\; times} \cdots, u): \mathcal Z^{n+1} \To  \mathcal Z^{n+1}$$ is a  complex structure on $\mathcal Z^{n+1}$. \end{prop}
\begin{proof} One only has to check that $U$ sends $\bigcap_{i=0}^{i=n} \ker \delta_\theta^i$ into itself: If $\delta_\theta^n(f)=0$ for $n=0,1,...,n$ then $U(f)(\theta) = u(f(\theta))= u(0)=0$, which is the case $n=0$. For $n=1$,
$\delta_\theta^1(Uf)= u(f)\delta_\theta^1(f) =0$. And then proceed by induction.\medskip

The induced complex structure can be made explicit by making explicit the spaces $\mathcal Z^{n+1}$: In Rochberg's representation they become $\mathcal Z^{n+1} = \{\big(t_n\delta_\theta^n(f), \dots, t_1 \delta_\theta^1(f), t_0 \delta_\theta (f)\big)\,
:\, f\in  \mathcal{H} \}$ for adequately chosen scalars $t_n$ (so that $t_n\delta_\theta^n(f)$ become the Taylor coefficients of $f$). The second assertion immediately follows from that.\end{proof}

\begin{cor} Every twisted sum space $d_\Omega L_2$ induced by a centralizer obtained from a scale of rearrangement invariant K\"othe function spaces admits a complex structure (compatible with the natural complex structure on $L_2$).\end{cor}
\begin{proof} The formula (\ref{rispace}) yields a complex structure on a scale of rearrangement invariant K\"othe spaces.\end{proof}

Propositions \ref{comint} and \ref{comintn} apply to more general realizations of Hilbert spaces. For instance, the couple $(B(H), S_1)$ yields the interpolation scale of Schatten classes $S_p$, which admits the following complex structure: if  $T = \sum  a_n(T) f_n \otimes e_n$ then
 $$\overline \omega (T) = \sum \omega( a_n(T)) f_n \otimes e_n.$$

Therefore, the Schatten-Kalton-Peck space $d_{\mathcal K} S_2$ induced by that scale admits he complex structure $(\overline\omega, \overline\omega)$. More general constructions in the domain of noncommutative $L_p$ spaces are possible (see \cite{ccgs,sua,cabetrue}); the corresponding noncommutative results will appear elsewhere.

\section{Complex structures on twisted Hilbert spaces}

We particularize now the situation described in the preceding section to the case in which one is working with scales of K\"othe function (resp. sequence) spaces in their $L_\infty$-module structure (resp. $\ell_\infty$-module structure) and $X_\theta$ is a Hilbert space; which is a rather standard situation since under  mild conditions the complex interpolation space $(X, \overline{X}^*)_{1/2}$  is a Hilbert space (see \cite{cfg}). One has:

\begin{teor}\label{extension} $\;$
\begin{enumerate}
\item For every non-trivial $L_\infty$-centralizer $\Omega$ on $L_2$  there is a complex structure on $L_2$ that can not be extended to an operator on  $d_{\Omega} L_2$.
 \item For every non-trivial $\ell_\infty$-centralizer $\Omega$ on $\ell_2$  there is a complex structure on $\ell_2$ that can not be extended to an operator on  $d_{\Omega} \ell_2$.
     \end{enumerate}
\end{teor}

Somehow, (2) is more complicated than (1), although both proofs follow from a criteria that will be established in Lemma \ref{nabla}. It will be helpful to recall that an operator between Banach spaces is said to be \emph{strictly singular}
if no restriction to an infinite dimensional closed subspace is an isomorphism. In this line, a quasi-linear map (in particular, a centralizer) is said to be \emph{singular} if its
restriction to every infinite dimensional closed subspace is never trivial. It is well known \cite{castmorestrict} that a  quasi-linear map is singular if and only if the associated exact sequence has strictly singular quotient map. Singular quasi-linear maps have been studied in
e.g.  \cite{ccs}, \cite{cfg}. The Kalton-Peck centralizer on $\ell_2$ is singular
\cite{kaltpeck}, while the Kalton-Peck centralizer on $L_2$ is not  \cite{cabe} (although it is "lattice" singular, according to the terminology of \cite{cfg}).

\adef Let $Z,Y$ be  Banach spaces and $\Omega: Z \lop Y$ be a quasi-linear map.
Given  a finite sequence $b=(b_k)_{k=1}^n$ of  vectors  in   $Z$ we will call $\nabla_{[b]} \Omega$  the number
$$ \nabla_{[b]} \Omega   ={\rm Ave}_{\pm} \left \|  \Omega  \left ( \sn  \pm b_k\right ) - \sn \pm \Omega(b_k)\right \|_Y,
$$
where the average is  taken over all  the signs $\pm 1$.
\zdef

Note that the quasi-linear condition guarantees that although $\Omega$ may take values in some $Y_0 \supset Y$, the differences appearing in the definition of $\nabla_{[b]} \Omega$ do belong to $Y$.
The triangle inequality holds for $\nabla_{[b]} \Omega$: if $\Omega$ and $\Psi$ are quasi-linear, then
$\nabla_{[b]} (\Omega + \Psi) \leq \nabla_{[b]} \Omega + \nabla_{[b]} \Psi.$ If $\lambda=(\lambda_k)_k$ is a finite sequence of scalars and  $x=(x_k)_k$  a sequence of vectors of $X$, we write $\lambda x$ to denote the finite sequence obtained by  the  non-zero vectors of $(\lambda_1x_1, \lambda_2x_2, \ldots )$. Recall that $\ell(Z,Y)$ is the vector space of linear (not necessarily continuous) maps $Z\to Y$, and that the distance between two homogeneous maps $T,S$ is the usual operator norm of the difference  $\|T-S\|$.

\begin{lemma} \label{trivial}Let $H$ be a Hilbert space and $\Omega: H \lop H$  be a quasi-linear map.
Assume that the restriction of $\Omega$ to a closed subspace $W\subset H$ is trivial. Then there exists $C \geq 1$ such that
for every  sequence $x=(x_k)_k$  of  normalized vectors in $W$ and every $\lambda=(\lambda_k)_k \in c_{00}$,
$$   \nabla_{[\lambda x]} \Omega\leq C  \|\lambda\|_2.
$$
% 2\dist(\Omega_{|W}, \ell(W,H))
\end{lemma}
\begin{proof} Suppose  that  $\Omega|_W$ is trivial for a closed subspace $W$ of $H$. Then we  can write $\Omega|_W=B+L$  with $B$ bounded homogeneous and $L$ linear. Let $x=(x_k)_{k=1}^n$ be a sequence of normalized vectors in $W$. Since $H$ has type 2 one gets:
\begin{eqnarray*}
 \nabla_{[\lambda x]} \Omega   &=& {\rm Ave}_{\pm} \left \|  B \left ( \sum_{k=1}^{\infty} \epsilon_k \lambda_k   x_k\right ) - \sum_{k=1}^{\infty} \epsilon_k B(\lambda_k  x_k)\right \| \\
&\leq& \|B\|{\rm Ave}_{\pm} \left \| \sum_{k=1}^{\infty}  \epsilon_k \lambda_k  x_k \right \| +Ave_{\pm}  \left \|  \sum_{k=1}^{\infty}  \epsilon_kB \lambda_k  x_k \right \| \\
&\leq& \|B\| \left(  \sum_{k=1}^{\infty}  \ \lambda_k^2  \|x_k\| ^2\right )^{1/2} +  \left(  \sum_{k=1}^{\infty} \|B \lambda_k  x_k\| ^2\right )^{1/2}\\
&\leq& 2\|B\|\left(  \sum_{k=1}^{\infty}  |\lambda_k  |^2\right )^{1/2}.
\end{eqnarray*}
\end{proof}

Observe that if  a sequence $x=(x_i)_i$  of $H$ is such that for a fixed constant $C$ and every $(\lambda_i)_i \in c_{00}$
$$\|\Omega(\sum_{i=1}^\infty \lambda_i x_i) - \sum_{i=1}^\infty \lambda_i \Omega(x_i)\| \leq C \|\sum_{i=1}^\infty \lambda_ix_i\|,$$
then the restriction of $\Omega$ to $[x_i]_i$ is trivial; actually the linear map $L( \sum_{i=1}^n \lambda_ix_i)=  \sum_{i=1}^n \lambda_i \Omega(x_i)$ is at distance $C$ from $\Omega$ on $[x_i]_i$. In other words, the condition of Lemma \ref{trivial} ``without averaging " is sufficient for triviality on $W=[x_i]_i$. When $\Omega$ is a centralizer, the original condition is necessary and sufficient:

\begin{lemma}\label{bounded} Let $\Omega$ be a centralizer  on $L_2$ (resp. $\ell_2$) and let $u=(u_n)_n$ be a disjointly supported  normalized sequence of vectors in $L_2$ (resp. $\ell_2$). The restriction of $\Omega$ to the closed linear span of the $u_n$'s   is trivial if and only if there is a constant $C>0$ such that
$$   \nabla_{[\lambda u]} \Omega \leq C \| \lambda \|_2$$
for every $\lambda=(\lambda_k)_k\in c_{00}$.
\end{lemma}

\begin{proof} If  the restriction of $\Omega$ to $[u_n]$   is trivial the result follows from Lemma \ref{trivial}.  Conversely, set  $v_i=\lambda_i u_i$, $v=\sum_{i=1}^n v_i$, and $\sum_{i=1}^n \epsilon_i v_i = \epsilon v$ for some function $\epsilon$ taking values $\pm 1$, thus the  centralizer $\Omega$ verifies

\begin{eqnarray*} \|\Omega(\epsilon \sum_{i=1}^n  v_i) - \epsilon \sum_{i=1}^n \ \Omega(v_i)\|
 &\leq & \|\Omega(\epsilon v) - \epsilon \Omega (v)\| + \| \epsilon \Omega( v) -  \epsilon \sum_{i=1}^n \Omega(v_i)\|\
  \end{eqnarray*}

%(recall that also $\Omega(v_i)$ are disjointly supported)
 For some constant $c>0$,

\begin{eqnarray*} \|\Omega(\epsilon \sum_{i=1}^n v_i) - \epsilon \sum_{i=1}^n \Omega(v_i)\|
&\leq& c\| v\| + \|\Omega (\sum_{i=1}^n \lambda_i u_i) - \sum_{i=1}^n \lambda_i \Omega(v_i)\|
  \end{eqnarray*}

Being the same true for all signs and all $\lambda_i$'s, one gets by taking the average

$$\|\Omega(\sum_{i=1}^n \lambda_i u_i) - \sum_{i=1}^n \lambda_i \Omega(u_i)\| \leq c \| \lambda \|_2  +   \nabla_{[\lambda u]} \Omega,$$
which yields the result.
\end{proof}

Let us recall from Lemma \ref{crit} that we say that an operator $u$ on a Hilbert space $H_2$ {\em can be lifted} to an operator on $H_1 \oplus_\Omega H_2$ if there exists an operator $\beta$ on $H_1 \oplus_\Omega H_2$ so that  $\beta_{|H_1}: H_1\to H_1$. The next lemma yields a method to construct a  complex structure on $H$ which cannot be lifted to operators on   $H\oplus_{\Omega} H$, provided there exist two separated sequences in $H$ with sufficiently different  $\nabla \Omega$. Precisely,

\begin{lemma}\label{nabla} Let $H$ be Hilbert space and let $\Omega: H \lop H$ be  a quasi-linear map. Suppose that $H$ contains two  normalized sequences  of vectors $a=(a_n)_n$,  $b=(b_n)_n$, such that:
\begin{itemize}
\item[(i)] $(a_n)_n$ and $(b_n)_n$ are equivalent to the unit vector basis of $\ell_2$,
\item[(ii)]    $[a]\oplus [b]$ forms a direct sum in $H$,
\item[(iii)]    $$\sup_{\lambda \in c_{00}} \frac{\nabla_{[\lambda b]} \Omega }{\|\lambda\|_2+\nabla_{[\lambda a]} \Omega } =+\infty$$
\end{itemize} Then $H$ admits a complex structure $u$ that cannot be lifted to any operator on $H \oplus_{\Omega} H$. %In particular $w$ may not be extended to a complex structure on $H\oplus_{\Omega} H$ or on one of its hyperplanes.
\end{lemma}
\begin{proof} There is no loss of generality in assuming that $\Omega$ takes values in $H$ and $[a] \oplus [b]$ spans an infinite codimensional subspace of $H$. Thus write $H=[a] \oplus [b]\oplus Y$ and let $v:Y \to Y$ be any choice of complex structure on $Y$.  We define  a complex structure $u$ on $H$ by setting $u( a_n) = b_n$, $u(b_n) =-a_n$,  and $u_{|Y}=v$.\medskip

Note that if  $\tau: H\to H$ is any bounded linear operator then for every $\lambda\in c_{00}$ one has
\begin{equation}\label{six}\nabla_{[\lambda b]}\left ( \tau \Omega u^{-1}\right) \leq \|\tau\|  \nabla_{[\lambda b]}\left (\Omega u^{-1}\right) = \|\tau \| \nabla_{[\lambda a]}\Omega.\end{equation}

Now, assume that  $u$ lifts to a bounded operator $\beta$ on $H \oplus_{\Omega} H$, which means that there exists a bounded operator $\tau$ on $H$ such that $(\tau, u)$ is compatible with $\Omega$. Thus, by Lemma \ref{crit},  $\tau\Omega- \Omega u = B+ L$, with $B$ homogeneous bounded and $L$ linear. Therefore,
$\tau\Omega u^{-1} = \Omega + Bu^{-1} + Lu^{-1}$. By the triangular inequality it follows
\begin{equation}\label{seven} |\nabla_{[\lambda b]}\left ( \tau \Omega u^{-1}\right)
- \nabla_{[\lambda b]} \Omega|
\leq \nabla_{[\lambda b]} (\tau \Omega u^{-1}-\Omega)  = \nabla_{[\lambda b]}  B \leq 2\|B\|\|\lambda\|_2,\end{equation}
for all $\lambda \in c_{00}$. Thus, the combination of (\ref{seven}) and (\ref{six}) yields,

\begin{equation*} \nabla_{[\lambda b]}  \Omega  \leq \|\tau\|\nabla_{[\lambda a]}  \Omega   + 2\|B\|\|\lambda\|_2, \end{equation*}
which contradicts (iii).\end{proof}

We shall need the next perturbation lemma, based on essentially well-known ideas in the theory of twisted sums.

\begin{lemma}\label{blocks} Let $Y$, $X$, $Z$ be Banach spaces. Let $\Omega: Z \to Y$ be a quasi-linear, and let $N: X \to Z$ be a nuclear operator of the form $N=\sum_n s_n x_n^*(\cdot) z_n$, $x_n^*$ normalized in $X^*$, $z_n$ normalized in $Z$, $s=(s_n)_n \in \ell_1$. Then $\Omega N: X \to Y$ is $C$-trivial for $C=Z(\Omega) \|s\|_1$. In the case when $X=Y$  and $Z$ is $B$-convex then there is a constant $c_Z$, depending only on $Z$, such that also $N\Omega: Z \to Y$ is $c_Z Z(\Omega) \|s\|_1$-trivial.
\end{lemma}

\begin{proof} The operator $N: X \to Z$ factorizes as $N = zDf$ where $f: X \to \ell_\infty$ is $f(x) = (x_m^*(x))_m$,  $D: \ell_\infty  \to \ell_1$ is
$D(\xi)=(\xi_m s_m)_m$ and  $z: \ell_1\to Z$ is $z(\xi)= \sum_{m} \xi_m z_m$. The fact that $N: X \to Z$ factorizes through $\ell_1$ guarantees that $\Omega N$ is trivial on $X$. Indeed if $L$ is the linear map defined on $c_{00}$ by $L(\xi)=\sum_m \xi_m \Omega(z_m)$, then $\|(\Omega z- L)(\xi)\| \leq Z(\Omega) \|\xi\|_1$.
Therefore for $x \in X$, $$\| \Omega N (x) - LDf (x) \| = \|(\Omega z - L) Df(x)\| \leq
Z(\Omega) \|Df(x)\|_1 \leq Z(\Omega) \|s\|_1 \|x\|.$$
We conclude by noting that $LDf$ is linear.
When $X=Y$ the fact that $N$ factorizes through $\ell_\infty$ guarantees that $N \Omega$ is trivial. Indeed according to \cite[Proposition 3.3]{kaltlocb}  each coordinate of the map $f \Omega: Z \to \ell_\infty$ may be approximated by a linear map $\ell_n: Z \to \R$ with constant $c_Z Z(\Omega)$, for some constant $c_Z$; therefore $f \Omega$ is at finite distance
$c_Z Z(\Omega)$ from a linear map $L: Z \to \ell_\infty$, and $N \Omega=zD f\Omega$ is at distance $c_Z Z(\Omega) \|s\|_1$ to $zDL$.
\end{proof}

The last ingredient we need is in the next lemma. Recall \cite{kaltloc} that a quasi-linear map $\Omega: X \to  Z$ is said to be {\em locally trivial} if there exists $C>0$ such that or any finite dimensional subspace $F$ of $X$, there exists a linear map $L_F$ such that $\|\Omega_{|F}-L_F\| \leq C$.

\begin{lemma}\label{disj}
Let $\Omega: X \to Z$ be a quasi-linear map, where $X$ is  a K\"othe space and $Z$ is reflexive. The following assertions are equivalent:
\begin{itemize}
\item[(i)] $\Omega$ is trivial.
\item[(ii)] $\Omega_{|Y}$ is trivial for any subspace $Y$ generated by disjoint vectors of $X$.
\item[(iii)] There exists $C>0$ such that for any finite dimensional subspace $F$ generated by disjoint vectors of $X$, there exists a linear map $L_F$ such that $\|\Omega_{|F}-L_F\| \leq C$.
\item[(iv)] $\Omega$ is locally trivial.
\end{itemize}
\end{lemma}
\begin{proof}
Assertions $(i)$ and $(iv)$ are well-known to be equivalent: trivial implies locally trivial while, see \cite{cabecastuni}, a locally trivial quasi-linear map taking values in a space complemented in its bidual is trivial. That $(i)$ implies $(ii)$ is obvious. Let us show that $(iii)$ implies $(iv)$: Let $\Omega$ be a quasi-linear map verifying $(iii)$ and let $F$ be a finite dimensional subspace of $X$. Approximating functions by characteristic functions we may find a nuclear operator $N$ on $X$ of arbitrary small norm so that $(Id+N)(F)$ is contained in the linear span $ [u_n]$ of a finite sequence of disjointly supported vectors. The restriction $\Omega_{|[u_n]}$ is trivial with constant $C$, thus using Lemma \ref{blocks},
$\Omega=\Omega(I+N)-\Omega N$ is trivial with constant $C+\epsilon$ on $F$.
Therefore $(iv)$ holds.\medskip

It remains to show that $(ii)$ implies $(iii)$. Let $\Sigma$ be the $\sigma$-finite base space on which the K\"othe function space $X$ is defined.  For a subset $A\subset \Sigma$ we will denote $X(A)$ the subspace of $X$ formed by those functions with support contained in $A$.\medskip

\textbf{Claim 1.} \emph{If $A$ and $B$ are disjoint and $\Omega$ is trivial on both $X(A)$ and $X(B)$ then it is trivial on $X(A \cup B)$.} Indeed, if $\|\Omega_{|X(A)}-a\| \leq c$ and
$\|\Omega_{|X(B)}-b\| \leq d$, where $a$ and $b$ are linear, then $\|\Omega_{|X(A \cup B)}-(a \oplus b)\| \leq 2(Z(\Omega)+c+d)$, where $a \oplus b$ is the obvious linear map on $X(A \cup B)$.\medskip

\textbf{Claim 2.} \emph{If $\Omega$ is nontrivial on $X$ then $\Sigma$ can be split in two sets $\Sigma = A \cup B$ so that $\Omega_{|X(A)}$ and $\Omega_{|X(B)}$ are both nontrivial}.  We first assume that $\Sigma$ is a finite measure space. Assume the claim does not hold. Split $\Sigma = R_1 \cup I_1$ in two sets  of the same measure and assume $\Omega_{|X(I_1)}$ is trivial. Note that since the claim does not hold, given any  $C \subset \Sigma$ and any splitting $C=A \cup B$ the map $\Omega$ is trivial on $X(A)$ or $X(B)$. So
 split $R_1  = R_2 \cup I_2$ in two sets of equal measure and assume that  $\Omega_{|X(I_2)}$ is trivial, and so on. If $\Omega$ is $\lambda$-trivial on $X(\cup_{j \leq n} I_j)$ for $\lambda<+\infty$ and for all $n$ then $\Omega$ is locally trivial on $X$ and therefore is trivial, a contradiction. If $\lambda_n\to \infty$ is such that $\Omega_{X(\cup_{j \leq n}I_j)}$ is $\lambda_n+1$-trivial but not $\lambda_n$-trivial for all $n$, then by the Fact we note that
for $m<n$, $\Omega$ cannot be trivial with constant less
 than $\lambda_n/2-Z(\Omega)-\lambda_m-1$ on $X(\cup_{m<j \leq n} I_j)$. From this we find a partition of $\N$ as $N_1 \cup N_2$ so that if $A=\cup_{n \in N_1} I_n$ and
$B=\cup_{n \in N_2} I_n$, then $\Omega$ is non trivial on $X(A)$ and $X(B)$, another contradiction.\medskip

 If $\Sigma$ is $\sigma$-finite then the proof is essentially the same: either one can choose the sets $I_n$ all having measure, say, $1$  or at some step $R_m$ is of finite measure, and we are in the previous case. This concludes the proof of the claim.\medskip

We pass to complete the proof that $(ii)$ implies $(iii)$. Assume that $\Omega$ is not trivial on $X$. By the claim, split $\Sigma = A_1 \cup B_1$ so that  $\Omega$ is trivial neither on $X(A_1)$ nor on $X(B_1)$. It cannot be locally trivial on them, so there is a finite number $\{u^1_n\}_{n\in F_1}$ of disjointly supported vectors on $X(A_1)$ on which $\Omega$ is not $2$-trivial. By the claim applied to $X(B_1)$ split $B_1= A_2 \cup B_2$ so that $\Omega$ is  trivial  neither on $X(A_2)$ nor in $X(B_2)$. It cannot be locally trivial on them, so there is a finite number of disjointly supported vectors $\{u^2_n\}_{n\in F_2}$ on $X(A_2)$ on which $\Omega$ is not $4$-trivial. Iterate the argument to produce a subspace $Y$ generated by an infinite sequence
$$ \{u^1_n\}_{n\in F_1} , \{u^2_n\}_{n\in F_2}, \dots, \{u^k_n\}_{n\in F_k}, \dots$$
of disjointly supported vectors, where $\Omega$ cannot be trivial.\end{proof}

We are thus ready to obtain:\medskip

\noindent \textbf{Proof of Theorem \ref{extension}(1)}: Let $\Omega: L_2 \lop L_2$ be an $L_{\infty}$-centralizer on $L_2$ and pick the dual centralizer $\Omega^*: L_2\lop L_2$: if $\Omega$ produces the exact sequence
$$\begin{CD}
0@>>> L_2 @>>> d_\Omega L_2 @>>> L_2 @>>> 0\end{CD}$$ then $\Omega^*$ is is the responsible to produce the dual exact sequence$$\begin{CD}
0@>>> L_2^* @>>> (d_{\Omega} L_2^*)^* @>>> L_2^* @>>> 0.\end{CD}$$

The existence and construction of $\Omega^*$ can be derived from general interpolation theory or from general theory of centralizers \cite{kaltmem,kaltdiff,cabeann}. A more cumbersome construction of the dual quasi-linear map is also possible and can be found in \cite{cabecastdu,castmoredu}.\medskip

We work from now on with $\Omega^*$ and the identification $L_2^*=L_2$. Using that $L_2^*[0,1]= L_2^*[0,1/2]\oplus L_2^*[1/2,1]$, we may assume without loss of generality that $\Omega^*|_{L_2^*[0,1/2]}$ is not trivial.  Then by Lemma \ref{disj},  $\Omega^*$ is  not trivial on  the span $[u_n]$  of  a normalized disjointly supported  sequence  $u=(u_n)_n$ in $L_2^*[0,1/2]$.
Using  Lemma  \ref{bounded} we obtain a sequence $(\lambda^n)_n$ of finitely supported sequences such that $ \nabla_{[\lambda^n u]} \Omega^*> n\|\lambda^n\|_2$ for every $n$.\medskip

On the other hand, it follows from \cite{cabe} that no centralizer on $ L_2^*$ is singular;  actually every such centralizer becomes trivial on the subspace generated by the Rademacher functions. Thus, there is a normalized  sequence $a=(a_n)_n$ in $L_2^*[1/2,1]$ equivalent  to $\ell_2$ such that $ \nabla_{[\lambda a]} \Omega^* \leq C \|\lambda\|_2$, for a fixed constant $C$ and every finitely supported sequence $\lambda$. Then we apply Lemma \ref{nabla} to get a complex structure $u$ on $L_2^*$ so that $(\tau, u)$ is not compatible with $\Omega^*$. Therefore $u^*$ is a complex structure on $L_2$ so that for every operator $\tau$ the couple $(u, \tau)$ is not compatible with $\Omega$.\quad $\Box$\\

We turn our attention to  sequence spaces. Recall that an operator $T$ between two Banach spaces $X$ and $Y$ is said to be {\em super-strictly singular} if every ultrapower of $T$ is strictly singular; which, by localization, means that
there does not exist a number $c>0$  and a sequence of subspaces $E_n\subseteq X$, $\dim E_n=n$, such that $\|Tx\|\geq c\|x\|$ for all $x\in \cup_n E_n$. Super-strictly singular operators have also been called finitely strictly singular; they were first introduced in \cite{M, M2}, and form a closed ideal containing the ideal of compact and contained in the ideal of strictly singular operators. See also \cite{css} for the study of such a notion in the context of twisted sums. An exact sequence $0\to Y \to X  \stackrel{q}\to Z\to 0$ (resp. a quasi-linear map) is called {\em supersingular} if $q$ is super-strictly singular. It follows from the general theory of twisted sums that an exact sequence induced by a map $\Omega: Z \to Y$  is supersingular if and only if for all $C>0$ there exists $n \in \N$ such that for all $F \subset Z$ of dimension $n$, and all linear maps $L$ defined on $F$ such that $\Omega_{|F} - L: F\to Y$ one has
$\|\Omega_{|F}-L\| \geq C$. We conclude the proof of Theorem \ref{extension}:\medskip

\noindent \textbf{Proof of Theorem \ref{extension}(2)}:  From the study in \cite{pliv} we know \cite[Thm 3]{pliv} that a super-strictly singular operator on a $B$-convex space has super-strictly singular adjoint. Therefore, no supersingular quasi-linear maps $\Omega: \ell_2\to \ell_2$ exist, since $B$-convexity is a $3$-space property (see \cite{castgonz}) and  the adjoint of a quotient map is an into isomorphism. Let now $\Omega$ be a centralizer on $\ell_2$, which cannot be supersingular. Let $e=(e_n)_n$ be the canonical basis of $\ell_2$.
We claim that there exists a decomposition $\ell_2=H \oplus H'$, where the two subspaces $H$ and $H'$ of $\ell_2$ are generated by subsequences $f$ and $g$ of $e$, such that $\Omega_{|H}$ is non trivial and $\Omega_{|H'}$ is not supersingular.
Indeed consider $\ell_2=H \oplus H'$ corresponding to the odd and even vectors of the basis.
If, for example, $\Omega_{|H'}$ is trivial then $\Omega_{|H}$ must be non trivial and we are done. So we may assume that $\Omega_{|H}$ and $\Omega_{|H'}$ are non trivial. By the ideal properties of super-strictly singular operators and the characterization of supersingularity, we note that either $\Omega_{|H}$ or $\Omega_{|H'}$ must be non supersingular. So up to appropriate choice of $H$ and $H'$, we are done.\medskip

Since $\Omega_{|H}$ is not trivial, Lemma  \ref{bounded} yields a sequence $f = (f_n)$ of disjointly supported vectors and a sequence $(\lambda^n)_n$ of vectors in $c_{00}$ such that $$ \nabla_{[\lambda^n f]} \Omega> n\|\lambda^n\|_2$$ for every $n$. On the other hand, since $\Omega_{|H'}$ is not supersingular, there exists $C$ and a sequence $G_n$ of $n$-dimensional subspaces of $H'$, with orthonormalized bases $g^n=(g^n_1,\ldots,g^n_n)$, such that $\Omega_{|G_n}=L_n+B_n,$ where $L_n$ is linear and $\|B_n\| \leq C$.
By Lemma \ref{trivial}, for all $\lambda=(\lambda_1,\ldots,\lambda_n)$ one has $\nabla_{[\lambda g^n]} \Omega \leq 2C \|\lambda\|_2$ and therefore
$$\nabla_{[\lambda^n g^n]} \Omega \leq 2C \|\lambda^n\|_2.$$
There is no loss of generality assuming that for different $n,m$ the elements of $g^n$ and $g^m$ are disjointly supported. Thus, pasting all those pieces we construct a sequence $g = (g^1, g^2, \dots, )$ of normalized vectors in $H'$, equivalent to the unit vector basis of $\ell_2$, and we note that $[f]$ and $[b]$ form a direct sum. The conclusion follows from Lemma \ref{nabla} in combination with the duality argument used in the proof of part (1). $\Box$

We conclude this section with several remarks about the previous results and some of their consequences.

\subsection{$\mathscr K$ is not supersingular.} The papers \cite{css,sua} contain proofs that the Kalton-Peck map $\mathscr K$ is not supersingular. A proof in the context of the techniques used in this paper can be presented:
 Consider the interpolation scale $\ell_2=(\ell_p, \ell_{p'})_{1/2}$ for a fixed $p$, where $1/p +1/p'=1$.  For every $n$ even let  us consider the finite sequence $f^n=(f_k^n)_{k=1}^n$ of orthonormalized vectors in $\ell_2$ defined by
\begin{eqnarray*}
f^n_1 &=& \frac{1}{\sqrt{2^n}} \sum_{j=1}^{2^n} (-1)^{j+1}e_{2^n+j-1} \\
f^n_2 &=& \frac{1}{\sqrt{2^n}} \sum_{j=1}^{2^{n-1}} (-1)^{j+1} ( e_{2^n+2(j-1)} + e_{2^n +2j-1})\\
&\vdots&\\
f^n_n &=& \frac{(e_{2^n}+ e_{2^n+1}+ \dotsb+ e_{2^n+2^{n-1}+1} ) - ( e_{2^n+2^{n-1}} + e_{2^n+2^{n-1}+1}+ \dotsb + e_{2^{n+1} -1} )}{\sqrt{2^n}}.
\end{eqnarray*}
For $x=\sn a_kf_k^n$ take the holomorphic function $g_x: \mathbb S \to \Sigma=\ell_p + \ell_{p'}$ given by $g_x(z)= 2^{n(1/2-1/p)(1-2z)} \sn a_kf_k^n$, which obviously satisfies $g_x(1/2)=  x$. To compute the norm of $g_x$ recall Khintchine's  inequality  \cite[Theorem 2.b.3]{lindtzaf}, which yields constants $A_p$, $B_p$ such that
 $$ A_p 2^{(1/p-1/2)n} \sqrt{ \sn \lambda_k^2} \leq \left\| \sn \lambda_k f^n_k \right\|_p \leq  B_p2^{(1/p-1/2)n} \sqrt{ \sn \lambda_k^2}$$
for all $n$ and scalars $\lambda_1, \ldots , \lambda_n$. Thus, $\|g_x\|_{\mathcal H} \leq c_p  \| x \|_2$ for some constant $c_p$ is independent of $n$ and the scalars $a_k$. Since $\mathscr K$ is the centralizer associated to the interpolation scale of the $\ell_p$'s, $\mathscr K = \delta'_{1/2}B$, where $B$ is a homogeneous bounded selection for the evaluation map $\delta_{1/2}$. Thus, by the estimate (\ref{select}), we have a constant $D_p$ such that
\begin{equation}\label{linear} \| \mathscr K (x) - g_x'(1/2)\| \leq D_p \|x\|_2.\end{equation}
 for $x$ in the span of $\{f^n_k\}$. Since $g_x'(1/2)= -2n\log n (1/2-1/p) \sn a_kf_k^n$, it is a linear map, which in particular shows that $\mathscr K$ is not supersingular.

\subsection{A computation of $\nabla \mathscr K$.} The use of supersingularity is necessary when $\Omega$ is singular,  since then no linear map approximates $\Omega$ on $[g] = [g^1, g^2, \dots]$. In the case of Kalton-Peck space, an explicit example of two sequences in which $\nabla \mathscr K$ behaves differently can be explicitly given: one is the canonical basis $a_n=(e_{n_k})_{k=1}^n$ where $\nabla_{[a_n]}\mathscr K = \frac{1}{2} \sqrt n \log n$ and the other the concatenation of the finite sequences $f^n=(f_k^n)_{k=1}^n$ on which $\nabla_{[f^n]}\mathscr K \leq D_p\sqrt{n}$ as it immediately follows from the inequality (\ref{linear}) above.

\subsection{Initial and final twisted Hilbert spaces.}
Theorem \ref{extension} answers a question in \cite{moretesis} raised by homological consideration about the existence of initial or final objects in the category of twisted Hilbert spaces. In its bare bones, the question is whether there is an \emph{initial} twisted Hilbert space
$$\begin{CD} 0@>>> \ell_2@>>> Z @>>> \ell_2 @>>>0;\end{CD}$$ i.e., a twisted Hilbert space such that for any other twisted Hilbert space
$$\begin{CD} 0@>>> \ell_2@>>> TH@>>> \ell_2 @>>>0\end{CD}$$ there is an operator $\tau: \ell_2 \to \ell_2$ and a  commutative diagram
$$\begin{CD} 0@>>> \ell_2@>>> Z @>>> \ell_2 @>>>0\\
&&@V{\tau}VV @VVV @|\\
0@>>> \ell_2@>>> TH @>>> \ell_2 @>>>0.\end{CD}$$
In particular it was asked whether the Kalton-Peck space could be an initial twisted Hilbert space. The dual question about the existence of a final twisted Hilbert space was also posed, meaning the existence of
a twisted Hilbert space $Z$ such that for any twisted Hilbert space $TH$, there exists $\tau: \ell_2 \rightarrow \ell_2$ and a commutative diagram
$$\begin{CD} 0@>>> \ell_2@>>> Z @>>> \ell_2 @>>>0\\
&&@|  @VVV @ V{\tau}VV\\
0@>>> \ell_2@>>> TH @>>> \ell_2 @>>>0.\end{CD}$$

A consequence of Theorem \ref{extension} is:

\begin{prop} No twisted Hilbert space induced by a centralizer on either $\ell_2$ or $L_2$ can be initial or final. \end{prop}

\begin{proof} Let $H$ denote either $\ell_2$ or $L_2$. Assume $H\oplus_\Omega H$ is the twisted Hilbert space induced by a nontrivial centralizer $\Omega$. Let $u$ be an isomorphism on $H$ that cannot be lifted to an operator on $H\oplus_\Omega H$ provided by Theorem \ref{extension}. This means that $\Omega u$ is never equivalent to  $\tau \Omega$ for no operator $\tau$. I.e., $\Omega$ is not initial. A similar argument, using an isomorphism on $H$ that cannot be extended to an operator on $H\oplus_\Omega H$ shows that $\Omega$ cannot be final.
\end{proof}

\subsection{About $\ell_2$-automorphy and $\ell_2$-extensibility.}
 Theorem \ref{extension} shows that no twisted sum $d_\Omega \ell_2$ or $d_\Omega L_2$ generated by a non-trivial centralizer can be {\em $\ell_2$-extensible}. A space $X$ is said to be $Y$-extensible if for every subspace $A$ of $X$ isomorphic to $Y$ every operator from  $A$ to $X$ can be extended to $X$. This notion was introduced in \cite{moreplic} in connection with the automorphic space problem of Lindenstrauss and Rosenthal \cite{lindrose}. A space $X$  is {\em automorphic  (resp. $Y$-automorphic)} when isomorphisms between subspaces of $X$ of same codimension (resp. and isomorphic to $Y$) can be extended to automorphisms of $X$. It was known \cite[p. 675]{cfm} that no non-trivial twisted Hilbert space can be automorphic and of course $Z_2$ cannot be $\ell_2$-automorphic  since $Z_2$ is isomorphic to its square. As a rule, $Y$-automorphic implies $Y$-extensible although the converse fails (see \cite{castmoreisr,moreplic,castplic,cfm} for details).

%\begin{prop} No twisted Hilbert space induced by a non-trivial centralizer is $\ell_2$-%extensible. \end{prop}

\section{Complex structures on $Z_2$}

 We focus now on the Kalton-Peck spaces; i.e., the twisted sums obtained from the scale of $L_p$-spaces with base space $[0,1]$ as $L_\infty$-modules, for which  $(L_\infty, L_1)_{1/2} = L_2$; and the scale of $\ell_p$-spaces as $\ell_\infty$-modules for which $(\ell_1, \ell_\infty)_{1/2} = \ell_2$. The associated centralizer in these two scales is $-2\mathscr K$, where
$$ \mathscr K (f) = f \log \frac{|f|}{\|f\|_2}.$$ Indeed, pick a bounded homogeneous selection for $\delta_{1/2}$, say $B_{1/2}(x)(z)=x^{2(1-z)}$  when $x\geq 0$  $\|x\|_2=1$, to get $\delta_{1/2}'B_{1/2}x= -2x\log x$ when $\|x\|=1$. The corresponding twisted sums will be called the Kalton-Peck spaces $Z_2 = d_{\mathscr K} \ell_2$ and $Z_2^F = d_{\mathscr K} L_2$, and likewise for the associated exact sequences. \medskip

In this section we will study in detail natural complex structures on $Z_2$. Recall that $\sigma: \N\to \N$ is the permutation $\sigma = (2,1)(4,3) \dots (2n, 2n-1) \dots $ and that $\omega$ is the linear map defined as
 $$\omega((x_n)) = ((-1)^n x_{\sigma(n)}).$$ From the results obtained so far we know that since the complex structure $\omega$ is an operator on the scale of $\ell_p$-spaces, $Z_2$ admits the complex structure $(\omega, \omega)$, which extends $\omega$. And also that there is a complex structure on $\ell_2$ that does not extend to an operator on $Z_2$. Analogous results hold for $Z_2^F$. There are actually at least three natural ways to consider complex structures on $Z_2$. The first, as we already said, is to simply consider those extended from complex structures on the scale of $\ell_p$, like $(\omega, \omega)$. We will call $Z_2^\omega$ the resulting space.
Secondly, the space $Z_2 \oplus Z_2$  can be endowed with the complex structure $U(x,y)=(-y,x)$. Since $Z_2 \simeq Z_2 \oplus Z_2$ this yields a complex structure on $Z_2$.  We will call this complex space $Z_2\oplus _\C Z_2$. There is a third way to define a natural complex space associated to $Z_2$: by directly forming the the complex Kalton-Peck space $Z_2(\mathbb C)$ \cite{kaltpeck} starting with the complex space $\ell_2(\mathbb C)$ and using the complex Kalton-Peck map $\mathscr K^\mathbb C(x) =x \log\frac{|x|}{\|x\|}$. It turns out that all of them are isomorphic.

\begin{prop}\label{complexifZ2} $\;$
 The spaces  $Z_2\oplus_{\mathbb C} Z_2$, $Z_2^{\omega}$ and $Z_2 (\mathbb{C})$ are isomorphic.
\end{prop}
\begin{proof} As usual, $\{e_n\}$ will be the canonical basis of $\ell_2$. Thus, if $x=\sum x_ne_n \in \ell_2$, then
$u(x)= \sum -x_{2n}e_{2n-1} + \sum x_{2n-1}e_{2n}.$ The isomorphism $T: Z_2^{\omega} \to Z_2\oplus_{\mathbb C}Z_2$
can be defined by the $\mathbb{C}$-linear map
$$T(x,y)=\left(  \left(\sum x_{2n-1}e_n, \sum y_{2n-1}e_n\right),  \left(\sum x_{2n}e_n, \sum y_{2n}e_n\right)\right),$$
for $(x, y)=(\sum x_ne_n, \sum y_ne_n)\in Z_2$. Let us write, $x_i= \sum x_{2n-1}e_n$ and $x_p= \sum x_{2n}e_n$. One has, $\|T(x,y)\|= \|((x_i,y_i), (x_p,y_p))\|= \|(x_i,y_i)\|+ \|(x_p,y_p)\|$.\\

Without loss of generality, suppose $\|y_i\|_2=1$ e $\|y_p\|_2\leq 1$. In this case $\|y\|_2^2= \|y_i\|_2^2 + \|y_p\|_2^2= 1+  \|y_p\|_2^2$. Now, since
$ \|(x_i,y_i)\| = \| x_i- \mathscr K y_i\|_2 + \|y_i\|_2$, for every $n$ we have,
\begin{eqnarray*}
|(x_i)_n- (\mathscr K y_i)_n| &=& \left|  x_{2n-1} - y_{2n-1} \log \frac{ \|y_i\|_2}{|y_{2n-1}|} \right|\\
&\leq& \left|  x_{2n-1} - y_{2n-1} \log \frac{ \|y\|_2}{|y_{2n-1}|} \right| + |y_{2n-1}| \left| \log \frac{ \|y\|_2}{|y_{2n-1}|}  -\log \frac{ \|y_i\|_2}{|y_{2n-1}|} \right|\\
&=&  \left|  x_{2n-1} - y_{2n-1} \log \frac{ \|y\|_2}{|y_{2n-1}|} \right| + |y_{2n-1}| \left| \log \|y\|_2 \right|\
\end{eqnarray*}
Therefore
\begin{eqnarray*}
\|(x_i,y_i)\|  &\leq&   \|(x,y)\|+ \|y_i\|_2 \left|\log \|y\|_2 \right| + \|y_i\|_2 \\
&=&   \|(x,y)\|+ \|y_i\|_2  \left( 1+ \frac{1}{2}\left|\log (1+\|y_p\|_2^2) \right| \right)\\
&\leq& \|(x,y)\| + (1+\log \sqrt 2) \|y\|_2\\
&\leq& (2+\log \sqrt 2 ) \|(x,y)\|
\end{eqnarray*}

Analogously, since $ \|(x_p,y_p)\| = \| x_p- \mathscr K y_p\|_2 + \|y_p\|_2$, we have for every $n$,
\begin{eqnarray*}
|(x_p)_n- (\mathscr K y_p)_n| &=& \left|  x_{2n} - y_{2n} \log \frac{ \|y_p\|_2}{|y_{2n}|} \right|\\
&\leq& \left|  x_{2n} - y_{2n} \log \frac{ \|y\|_2}{|y_{2n}|} \right| + |y_{2n}| \left| \log \frac{ \|y\|_2}{|y_{2n}|}  -\log \frac{ \|y_p\|_2}{|y_{2n}|} \right|\\
&=&  \left|  x_{2n} - y_{2n} \log \frac{ \|y\|_2}{|y_{2n}|} \right| + |y_{2n}| \left| \log \frac{\|y\|_2}{\|y_p\|_2} \right|\\
&=&\left|  x_{2n} - y_{2n} \log \frac{ \|y\|_2}{|y_{2n}|} \right| +  \frac{1}{2}|y_{2n}| \left| \log \left( 1+ \frac{1}{\|y_p\|_2^2}\right) \right|
\end{eqnarray*}

Therefore

\begin{eqnarray*}
\|(x_p,y_p)\|  &\leq&   \|(x,y)\|+ \frac{1}{2}\|y_p\|_2 \left| \log \left( 1+ \frac{1}{\|y_p\|_2^2}\right) \right|
 + \|y_p\|_2 \\
&=&   \|(x,y)\|+ \frac{1}{2} +  \|y_p\|_2\\
&\leq& \|(x,y)\| + 2 \|y\|_2\\
&\leq& 3 \|(x,y)\|
\end{eqnarray*}because $|t\log(1+1/t^2)|\leq 1$ for $0<t<1$. So, we conclude that $\|T(x,y)\|\leq (5+\log \sqrt 2)\|(x,y)\|$. The general case is immediate.\medskip

We show the isomorphism between $Z_2^\omega$ and $Z_2(\C)$.  Denote by $\{f_n\}$ the canonical basis for $\ell_2(\mathbb C)$ and define the complex isomorphism $A: \ell_2^\omega \rightarrow \ell_2(\mathbb C)$ given by
$$A(\sum x_nf_n)= \sum (x_{2n-1} + ix_{2n})e_n.$$

The (real) Kalton-Peck map $\mathscr K$ makes sense as a complex quasi-linear map $\ell_2^\omega\To \ell_2^\omega$ since it is $\mathbb C$-homogeneous: $\mathscr K (i\cdot x) = \mathscr K(\omega x) = \omega \mathscr K(x) = i\cdot \mathscr K(x)$. Let us show that there is a commutative diagram:

$$\begin{CD} 0@>>> \ell_2^\omega @>>> \ell_2^\omega \oplus_{\mathscr K} \ell_2^\omega @>>> \ell_2\omega@>>>0\\
&&@VAVV @V(A,A)VV @VVAV\\
0@>>> \ell_2(\mathbb C) @>>> \ell_2(\mathbb C) \oplus_{\mathscr K^c} \ell_2(\mathbb C) @>>> \ell_2(\mathbb C) @>>> 0;\end{CD}$$

for which we will check that $A\mathscr K - \mathscr K^cA$ is bounded on finitely supported sequences. Since

$$ A\mathscr K(x) =   A( \sum x_n\log \frac{\|x\|_2}{|x_n|} f_n) = \sum  \left( x_{2n-1}\log \frac{\|x\|_2}{|x_{2n-1}|}  + i x_{2n}\log \frac{\|x\|_2}{|x_{2n}|} \right)e_n$$
and

$$\mathscr K^c A(x) =   \mathscr K^c \left( \sum (x_{2n-1} + ix_{2n})e_n \right) = \sum  \left( (x_{2n-1}+ ix_{2n})\log \frac{\|x\|_2}{|x_{2n-1}+ i x_{2n}|}  \right)e_n$$

one gets

\begin{eqnarray*}
\| A\mathscr K(x)- \mathscr K^c A(x)\|^2  &=&  \left\| \sum \left( x_{2n-1} \log \frac{| x_{2n-1}+ix_{2n}|}{|x_{2n-1}|}  + ix_{2n} \log \frac{| x_{2n-1}+ix_{2n}|}{|x_{2n}|}\right)e_n\right\|^2\\
&=& \sum  |x_{2n-1} + ix_{2n}|^2 \left| \frac{ x_{2n-1}}{x_{2n-1} + ix_{2n}} \right|^2 \left| \log \frac{|x_{2n-1} + ix_{2n}|}{ |x_{2n-1}|}\right|^2\\
&+&  \sum  |x_{2n-1} + ix_{2n}|^2 \left| \frac{ x_{2n}}{x_{2n-1} + ix_{2n}} \right|^2 \left| \log \frac{|x_{2n-1} + ix_{2n}|}{ |x_{2n}|}\right|^2\\
&\leq& 2\sum |x_{2n-1} + ix_{2n}|^2\\
&=& 2\|x\|_2^2.
\end{eqnarray*}
\end{proof}

\begin{cor}
$\;$
 \begin{enumerate}
\item For any complex structure $u$ on $Z_2$, the space $Z_2^u$ is isomorphic to a complemented subspace of $Z_2(\mathbb C)$.
\item For any complex structure $u$ on $Z_2$, the space $Z_2^u$ is $Z_2(\mathbb C)$-complementably saturated and $\ell_2(\mathbb C)$-saturated.
\end{enumerate}
\end{cor}
\begin{proof} The first assertion is a general fact: given a complex structure $u$ on some real space $X$, $X^u$ is complemented in $X\oplus_{\mathbb C} X$: indeed, $X \oplus_{\mathbb C} X = \{(x, ux), x \in X\} \oplus \{(x, -ux), x \in X\}$,
the first summand being $\mathbb C$-linearly isomorphic to $X^{-u}$, and the second to $X^{u}$. The second assertion follows from results of Kalton-Peck \cite{kaltpeck}: complex structures on $Z_2$ inherit saturation properties of the complex $Z_2(\mathbb C)$.
\end{proof}

\begin{prop}\label{uniq} Let $u$ be a complex structure on $Z_2$. If $Z_2^u$ is isomorphic to its square then it is isomorphic to $Z_2(\mathbb C)$.\end{prop}

\begin{proof} The space $Z_2(\mathbb C)$  is isomorphic to its square. Use the Corollary above and Pe\l czy\'nski's decomposition method.\end{proof}

\section{Compatible complex structures on the hyperplanes of $Z_2$}

As we mentioned in the introduction, our main motivation is the old open problem of whether $Z_2$ is isomorphic to its hyperplanes.
Since $Z_2$ admits complex structures, showing  that hyperplanes of $Z_2$ do not admit complex structures would show that
$Z_2$ is not isomorphic to its hyperplanes. What we are going to show is that hyperplanes of $Z_2$ do not admit complex structures compatible, in a sense to be explained, with its twisted Hilbert structure. More precisely, we already know the existence of complex structures on $\ell_2$ that do not extend to complex structures (or even operators) on $Z_2$, and we shall show in this section that no complex structure on $\ell_2$ can be extended to a complex structure {\em on a hyperplane of $Z_2$}. An essential ingredient to prove this is \cite[Proposition 8]{feregale}, which may be stated as:

\begin{prop}\label{FG} Let $u, h$ be  complex structures  on, respectively, an infinite dimensional real Banach space $X$ and some hyperplane $H$ of $X$.
Then the operator $u_{|H}- h$ is not strictly singular.
\end{prop}

This proposition may be understood as follows. It is clear that $u_{|H}$ may not be {\em equal} to $h$. For then  the induced quotient operator $\tilde{u} (x+H)= ux+H$ would be a complex structure on $X/H$, which has dimension $1$. The result above extends this observation to strictly singular perturbations.

The following result appears proved in \cite{cck}:

\begin{lemma}\label{3strict} Assume one has a commutative diagram:

$$\begin{CD}
0@>>> A @>>>  B @>q>> C @>>> 0\\
 &&@VtVV @VVTV @| \\
0@>>> D@>>>  E@>>> C@>>>0
\end{CD}$$
If both $q$ and $t$ are strictly singular then $T$ is strictly
singular.
\end{lemma}

Consequently, if one has an exact sequence $0\to Y\to X\stackrel{q}\to Z\to 0$ with strictly singular quotient map $q$ and an operator $t:Y\to Y'$ admitting two extensions $T_i: X\to X'$, ($i=1,2$), so that one has commutative diagrams
$$\begin{CD}
0@>>> Y @>>>  X @>q>> Z @>>> 0\\
 &&@VtVV @VVT_iV @| \\
0@>>> Y'@>>>  X'@>>> Z@>>>0
\end{CD}$$
then necessarily $T_1-T_2$ is strictly singular simply because $(T_2 - T_1)|_Y = 0$ and thus one has the commutative diagram
$$\begin{CD}
0@>>> Y @>>>  X @>q>> Z @>>> 0\\
 &&@V0VV @VV{T_1-T_2}V @| \\
0@>>> Y'@>>>  X'@>>> Z@>>>0.
\end{CD}$$

From this we get:

\begin{prop}\label{hyper} Let  $0\to Y\to X\stackrel{q}\to Z\to 0$ be an exact sequence of real Banach spaces with strictly singular quotient map. Let $u$ be a complex structure on $Y$ and let $H$ be a hyperplane of $X$ containing $Y$. Then $u$ does not extend simultaneously to a complex structure on $X$ and to a complex structure on $H$.
\end{prop}

\begin{proof}
Let $\gamma$ be a complex structure on $X$ such that $\gamma_{|_Y} = u$. Assume that $\gamma^H$ is a complex structure on $H$ extending also $u$, then it follows from the previous argument that $\gamma_{|_{H}} - \gamma^H$ is strictly singular; which contradicts Proposition \ref{FG}.
\end{proof}

And therefore, since the Kalton-Peck map is singular:

\begin{cor}\label{hyper2} Let $u$ be any complex structure on $\ell_2$ and let $H$ be a hyperplane of $Z_2$ containing the canonical copy of $\ell_2$ into $Z_2$. If $u$ extends to a complex structure on $Z_2$ then it cannot extend to a complex structure on $H$.
\end{cor}

We shall show that the conclusion in the corollary always holds. We shall actually prove this in a slightly more general setting by isolating the notion of complex structure on an hyperplane of a twisted sum {\em compatible} with the twisted sum.

\adef Let $0 \to Y\to  X\to Z\to 0$ be an exact sequence with associated quasi-linear map $\Omega$. We will say that a complex structure $u$ on an hyperplane $H$ of $X$ is compatible with $\Omega$, or simply compatible, if  $u(H\cap Y)\subset Y$. \zdef

 In other words, $u$ is compatible if $u$ restricts to a complex structure on $H \cap Y$.
Of course if the hyperplane $H$ contains  $Y$ then this is simply restricting to  a complex structure on  $Y$. Our purpose  is now to show:

\begin{teor}\label{comphyp}
No complex structure on an hyperplane of $Z_2$ is compatible with Kalton-Peck map  $\mathscr K$.
\end{teor}

We need some extra work.  Given a quasi-linear map $\Omega : Z \lop Y$ and a closed subspace $H=[u_n] \subset Z$ generated by a basic sequence, we define the following quasi-linear map, associated to $H$, on finite combinations $x = \sum \lambda_n u_n$,
$$\Omega_H(\sum \lambda_n u_n)= \Omega(\sum \lambda_n u_n) - \sum \lambda_n \Omega u_n.$$ Since the difference $\Omega_{|H}-\Omega_H$ is linear, we have that the quasi-linear maps $\Omega_{|H}$ and $\Omega_H$ are equivalent.\\

%In particular, for any operator $T: X\to X$ and $x = \sum \lambda_i u_i$ , $[\Omega, T]_{|H}(x) - [\Omega, T]_H(x) = \sum \lambda_i[\Omega, T](u_i)$ is linear.

Assume that $T$ is an isomorphism on $Z$, then $(Tu_n)$ is again a basic sequence and $TH = [Tu_n]$, so the meaning of $\Omega_{TH}$ is clear. One thus has $[T,\Omega]_{|H} \equiv T\Omega_H - \Omega_{TH}T_{|H}$.

The following proposition will be useful.

\begin{prop}\label{strictcommutator} For every operator $T: \ell_2\to \ell_2$, and for every block subspace $W$ of $X$, the commutator $[\mathscr K, T]$ is trivial on some block subspace of $W$. In particular, $[\mathscr K,T]$ is not singular.
\end{prop}

\begin{proof} Assume that $T$ is an isomorphism on a block subspace $W = [u_n]$. By Lemma \ref{blocks} we may replace $T$ by a nuclear perturbation to assume also that $\{Tu_n\}$ is a block basis (jumping to a subsequence if necessary). It is not hard to check (see also \cite[Lemma 3]{ccs}) that ${\mathscr K}_{W}$ is a linear perturbation of the Kalton-Peck map of $W$; i.e. for $x=\sum \lambda_nu_n$ one has
(identities are up to a linear map):
$${\mathscr K}_W(x) \equiv \sum \lambda_n \log \frac{\|x\|}{|\lambda_n|}u_n$$
and
$${\mathscr K}_{TW}(Tx) \equiv \sum \lambda_n \log \frac{\|Tx\|}{|\lambda_n|}Tu_n.$$
Thus,$$[\mathscr K, T]_{|W}(x) \equiv \sum \lambda_n \log \frac{\|x\|}{|\lambda_n|}Tu_n - \sum \lambda_n \log \frac{\|Tx\|}{|\lambda_n|}Tu_n = T( \sum \lambda_n \log \frac{\|x\|}{\|Tx\|}u_n).$$
%$$\sum \lambda_n \log \frac{\|x\|}{\|Tx\|}Tu_n = T\left( \sum \lambda_n \log \frac{\|x\|}{\|Tx\|}u_n \right)  $$
Let $c \geq 1$ be such that $c^{-1} \|x\|\leq \|Tx\| \leq c \|x\|$ on $W$; since $[u_n]$ is unconditional
$$\left\| T \left( \sum \lambda_n \log \frac{\|x\|}{\|Tx\|}u_n \right)\right \| \leq \|T\|\left \|\sum \lambda_n \log \frac{\|x\|}{\|Tx\|}u_n\right \| \leq (\|T\|\log c) \|x\|.$$
So the commutator $[\mathscr K, T]$ is trivial on $W$.\\

It follows immediately that if $[\mathscr K,T]$ is singular on some $W$,  then  $T$ is strictly singular on $W$. But since $[\mathscr K, T]\equiv - [\mathscr K, I - T]$,
the previous implication means that both $T$ and $I-T$ are strictly singular on $W$, which is a contradiction.
\end{proof}

%\noindent \textbf{Remark}. Observe that what has been really proved is that $\mathscr K$ has %the following property: if $T$ is an operator on $\ell_2$ which is an isomorphism on a block %subspace $W$ of $\ell_2$, then $[\mathscr K,T]$ is trivial on some block-subspace of $W$.

\begin{prop}\label{compact} Let $X$ be a space with a basis and $\Omega: X\to X$ be a  singular quasi-linear map, with
the property that for any operator $T$ on $X$ and for every  block subspace $W$ of $X$,  $[T,\Omega]$ is trivial on some block subspace of $W$. Let $T, U: X\to X$ be  bounded linear operators such that $(T,U)$ is compatible with $\Omega$.  Then $T-U$ is strictly singular.
\end{prop}
\begin{proof} Assume that the restriction of $T-U$ to some block subspace $W$ is an isomorphism into, then passing to a subspace we may assume that $\Omega T - T \Omega$ is trivial on $W$.
 Since
$\Omega U - T\Omega$ is trivial by hypothesis and $\Omega U - \Omega T - \Omega (U - T)$ is trivial, we obtain that $\Omega(T-U)$ is trivial on $W$. This means
that $\Omega$ is trivial on $(T-U)W$, which is not possible because $\Omega$ is singular. So $T-U$ must be strictly singular.
\end{proof}

\begin{cor}\label{compactZ2}  If $T, U$ are bounded linear operators on $\ell_2$ such that
$(T,U)$  is compatible with $\mathscr K$, then  $T-U$ is compact.
\end{cor}

We are ready to obtain:\\

\noindent \textbf{Proof of Theorem \ref{comphyp}} Let $H$ be an hyperplane of $Z_2$ admitting a compatible complex structure $u$. Starting from the representation
$$\begin{CD}
0@>>> \ell_2@>>> Z_2 @>q>> \ell_2 @>>> 0\end{CD}$$
 there are two possibilities: either $H$ contains (the natural copy of) $\ell_2$ or not. If $H$ contains $\ell_2$  then $u_0: \ell_2 \to \ell_2$, defined by $u_0(x)=u(x)$ is also a complex structure. This induces a complex structure $\overline{u}: H/\ell_2 \to H/\ell_2$ yielding a commutative diagram:

$$\begin{CD}
0@>>> \ell_2 @>>> H @>>> H/\ell_2@>>> 0\\
&&@V{u_0}VV @VVuV @VV{\overline u}V\\
0@>>> \ell_2 @>>> H @>>> H/\ell_2@>>> 0.
\end{CD}$$

Observe that if $j: H/\ell_2 \to \ell_2$ is the canonical embedding then the quasi-linear map associated to the above exact sequence is $\mathscr K j$; and the commutative diagram above means that the couple $(u_0, \overline u)$ is compatible with $\mathscr K j$. Since $j(H/\ell_2)$ is an hyperplane of $\ell_2$, we can extend $j\overline{u}j^{-1}$ to an operator $\tau: \ell_2\to \ell_2$, and the couple $(u_0, \tau)$ is compatible with $\mathscr K$: the maps $u_0\mathscr K$ and $\mathscr K \tau$ are equivalent since so are their restrictions
$u_0\mathscr K j$ and $\mathscr K \tau j = \mathscr K j \overline{u}$ to a one codimensional subspace.  By Corollary \ref{compactZ2}, $u_0 - \tau$ is compact.

 On the other hand, we have a complex structure $u_0$ on $\ell_2$ and another $j \overline u j^{-1}$ on an hyperplane of $\ell_2$, which means, according to Proposition \ref{FG}, that $ u_0 j - j \overline u$ cannot be strictly singular; hence $u_0 - \tau$ cannot be strictly singular either; yielding a contradiction.\\

If, on the other hand, $H$ does not contain $\ell_2$, then necessarily $\ell_2 + H =  Z_2$ and thus $H/ (\ell_2\cap H) \simeq (H + \ell_2)/\ell_2 \simeq Z_2/\ell_2 \simeq \ell_2$ and one has the commutative diagram$$\begin{CD}
0@>>> \ell_2\cap  H@>>> H @>>> \ell_2 @>>> 0\\
&&@ViVV @VVV @|\\
0@>>> \ell_2 @>>>  Z_2@>>>  \ell_2 @>>> 0 \\
&&@VVV @VVV\\
&&\R @= \R.
\end{CD}$$
which means that there is a version $\mathscr K'$ of $\mathscr K$ whose image is contained in $\ell_2\cap H$; i.e., $\mathscr K \equiv i \mathscr K'$.\medskip

Let again $u$ be a complex structure on $H$ inducing a complex structure $u_0$ on $\ell_2\cap  H$, and therefore a complex structure $\overline u$ on $\ell_2$ so that there is a commutative diagram
 $$\begin{CD}
0@>>> \ell_2\cap  H@>>> H @>>> \ell_2 @>>> 0\\
&&@Vu_0VV @VVuV @VV{\overline u}V\\
0@>>> \ell_2\cap H @>>>  H@>>>  \ell_2 @>>> 0.
\end{CD}$$

Therefore the couple $(u_0, \overline u)$ is compatible with $\mathscr K'$. Extend $u_0$ to an operator $\tau$ on $\ell_2$ to get a couple $(\tau, \overline u)$ compatible with $\mathscr K$, which yields that $\tau - \overline u$ is compact. On the other hand, $u_0$ is a complex structure on  $\ell_2\cap H$, while $\overline u$ is a complex structure on $\ell_2$ which, by Proposition \ref{FG} implies that $u_0 - \overline u_{|\ell_2\cap H} $ cannot be strictly singular, so $\tau - \overline u$ cannot be strictly singular either, a contradiction.$\Box$

\section{Open problems}

Of course, the main question left open in this paper is

\begin{prob} Do hyperplanes of $Z_2$ admit complex structures?
\end{prob}

And the second could be

\begin{prob} Is there a twisted Hilbert space not admitting a  complex structure?
\end{prob}

Complex and compatible complex structures could be different objects. One could daydream about to what extent all complex structures on $Z_2$ are compatible with some nice representation.

Passing to more specific issues, observe that Proposition \ref{strictcommutator} depends on the explicit form of the Kalton-Peck map. Is this result true for arbitrary  singular quasi-linear maps? Is Corollary \ref{compactZ2}  true in the general case? Precisely

\begin{prob} Assume $\Omega: \ell_2 \to \ell_2$ is a singular quasi-linear map, and that $T,U$ are bounded linear operators on $\ell_2$ such that $(T,U)$ is compatible with $\Omega$. Does it follow that $T-U$ is compact?
\end{prob}

Or, else, keeping in mind that $Z_2$ admits many representations as a twisted sum space, even of non-Hilbert spaces:

\begin{prob} Given an exact sequence $0 \to H \to Z_2 \to H \to 0$, with $H$ Hilbert, is it true  that given $T,U$ bounded linear operators on $H$ such that $(T,U)$ is compatible with the sequence then $T-U$ must be compact?
\end{prob}

\noindent Or, (this formulation is due to G. Godefroy):

\begin{prob} Let $\Omega$ be a (singular) centralizer on $\ell_2$. Do hyperplanes of $d_\Omega \ell_2$ admit compatible complex structures?
\end{prob}

A few other questions specific to Kalton-Peck space remain unsolved:
the first one may give a finer understanding of complex structures on $Z_2$.

\begin{prob} Assume that $(T,U)$ is compatible with Kalton-Peck map $\mathscr K$. Does there exist a compact perturbation $V$ of $T$ (and therefore of $U$) such that $(V,V)$ is compatible with $\mathscr K$? \end{prob}

Two more related although more general problems can be formulated. We conjecture a positive solution to the following problem:

\begin{prob} Show that $Z_2$ admits unique complex structure. \end{prob}

The answer could follow from Proposition \ref{uniq}. In connection with this we do not know the answer to:

\begin{prob} Find a complex space which is isomorphic to its square as a real space but not as a complex space. \end{prob}


\begin{thebibliography}{99}


%\bibitem{BeLi} \emph{Y. Benyamini and J. Lindenstrauss},  Geometric
%nonlinear functional analysis, Vol. 1, Colloq. Publ. Amer. Math.
%Soc., Providence, 1999.

\bibitem{A} R. Anisca. \emph{Subspaces of $L_p$ with more than one complex structure,} Proc. Amer. Math. Soc. 131 (2003) 2819--2829.

\bibitem{AH} S. A. Argyros, R. G. Haydon.
\emph  {A hereditarily indecomposable $\mathfrak L_{\infty}$-space that solves the scalar-plus-compact problem,} Acta Math. 206 (2011), no. 1, 1--54.

\bibitem{B} S. Banach. \emph {Th\'eorie des op\'erations lin\'eaires,} Monografje Math., Warszawa (1932).

%\bibitem{BL} Y . Benyamini,  J. Lindenstrauss. \emph  {Geometric Nonlinear Functional Analysis, Vol. 1,}
%Amer. Math. Soc. Colloq. Pub. Vol. 48 (2000).

\bibitem{Bergh-Lofstrom} J. Bergh and J. L\"ofstr\"om. \emph { Interpolation spaces. An introduction,}
Springer-Verlag, (1976).


\bibitem{cabetwisted} F. Cabello S\'{a}nchez. \emph{Twisted Hilbert spaces,}
Bull. Austral. Math. Soc. 59 (1999) 17--180.

\bibitem{cabe} F. Cabello S\'{a}nchez,
\emph{There is no strictly singular centralizer on $L_p$,} Proc. Amer. Math. Soc. 142 (2014) 949--955.

\bibitem{cabeann} F. Cabello S\'{a}nchez, \emph{Nonlinear centralizers in homology}, Math. Ann. 358 (2014) 779-798.

\bibitem{cabetrue} F. Cabello S\'{a}nchez, \emph{The noncommutative Kalton-Peck spaces}, preprint 2015.

\bibitem{cabecastdu} F.~Cabello S\'{a}nchez and J.M.F.~Castillo,
\emph{Duality and twisted sums of Banach spaces}, J. Funct. Anal. 175 (2000) 1--16.


\bibitem{cabecastuni} F. Cabello S\'{a}nchez, J.M.F. Castillo, \emph{Uniform boundedness and twisted sums
of Banach spaces}, Houston J. of Math. 30
(2004) 523--536.

\bibitem{ccgs} F. Cabello S\'{a}nchez, J.M.F. Castillo, S. Goldstein, J. Su\'arez, \emph{Twisting noncommutative $L_p$-spaces}, submitted.


\bibitem{cck}F. Cabello S\'{a}nchez, J.M.F. Castillo, N.J. Kalton, \emph{Complex interpolation and twisted twisted Hilbert spaces}, Pacific J. Math. 276 (2015) 287 - 307.

\bibitem{ccs} F. Cabello S\'anchez, J.M.F. Castillo, J. Su\'arez.
\emph { On strictly singular nonlinear centralizers,} Nonlinear Anal.- TMA 75 (2012)  3313--3321.

\bibitem{carro} M. J. Carro, J. Cerd\`{a}, J. Soria, \emph{Commutators and interpolation methods}, Ark. Math. 33 (1995) 199--216.

\bibitem{cfg} J.M.F. Castillo, V. Ferenczi and M. Gonz\'alez, \emph{Singular twisted sums generated by complex interpolation}, Trans. Amer. Math. Soc. to appear.

\bibitem{cfm} J.M.F. Castillo, V. Ferenczi and Y. Moreno, \emph{On Uniformly Finitely Extensible Banach spaces}, J. Math. Anal. Appl. 410 (2014) 670--686.

\bibitem{castgonz}  J.M.F. Castillo, M. Gonz\'alez.
\emph{ Three-space problems in Banach space theory,} Lecture Notes in Math. 1667, Springer  (1997).

\bibitem{castmoreisr} J.M.F. Castillo, Y. Moreno. \emph {On the Lindenstrauss-Rosenthal theorem,} Israel J. Math. 140 (2004) 253--270.

\bibitem{castmorestrict} J.M.F. Castillo and Y. Moreno,
\emph{Strictly singular quasi-linear maps,} Nonlinear Anal. - TMA 49 (2002) 897--904.


\bibitem{castmoredu} J.M.F. Castillo and Y Moreno, \emph{Twisted dualities for Banach spaces},
in Proceedings of "Banach spaces and their applications in
Analysis" in honour of Nigel Kalton (2007); Walter De Gruyter.


\bibitem{castplic} J.M.F. Castillo and A. Plichko,
\emph{Banach spaces in various positions.} J. Funct. Anal. 259
(2010) 2098-2138.

\bibitem{css} J.M.F. Castillo, M. Sim\~{o}es, J. Su\'arez.
\emph{On a Question of Pe\l czy\'nski about Strictly Singular Operators,} Bull. Pol. Acad. Sci. Math. 60 (2012) 27--36.


\bibitem{cjmr} M. Cwikel, B. Jawerth, M. Milman, R. Rochberg.
\emph{  Differential estimates and commutators in interpolation theory in  ``Analysis at Urbana II'', } London Mathematical Society, Lecture Note Series, (E. R. Berkson, N. T. Peck, and J Uhl, Eds.), Vol. 138,
Cambridge University Press, Cambridge, (1989) pp. 170--220.

\bibitem{C} W. Cuellar Carrera. \emph{A Banach space with a countable infinite number of complex structures,}
J. Funct. Anal. 267 (2014) 1462--1487.

\bibitem{D} J. Dieudonn\'e.
\emph {Complex structures on real Banach spaces,} Proc. Amer. Math. Soc. 3 (1952)  162--164.

\bibitem{F} V. Ferenczi.
\emph {Uniqueness of complex structure and real hereditarily indecomposable Banach spaces,}
Adv. Math. 213 (2007)  462--488.

\bibitem{feregale}  V. Ferenczi, E.  Galego.
\emph  {Even infinite-dimensional real Banach spaces}, J. Funct.  Anal. 253 (2007) 534--549.

\bibitem{FG} V. Ferenczi, E. Galego.
\emph {Countable groups of isometries on Banach spaces,} Trans. Amer. Math. Soc. 362 (2010) 4385--4431.

\bibitem{G} W.T. Gowers. \emph { A solution to Banach's hyperplane problem,} Bull. Lond. Math. Soc. 26 (1994) 523--530.

\bibitem{GM} W.T. Gowers, B. Maurey.
\emph {The unconditional basic sequence problem,} J. Amer. Math. Soc. 6 (1993) 851--874.

\bibitem{hiltstam} E. Hilton, K. Stammbach.
\emph{ A course in homological algebra,} GraduateTexts in Math. 4, Springer-Verlag (1970).


%\bibitem{JLS} W.B. Johnson, J. Lindenstrauss and G. Schechtman.
%\emph {On the relation between several notions of unconditional structure,}  Israel J. Math. 37 (1980) 120--129.


\bibitem{kaltloc} N.J. Kalton, \emph{Locally complemented subspaces and
${\mathcal L}_p$ for $p<1$}, Math. Nachr. 115 (1984) 71-97.

\bibitem{kaltmem}  N.J. Kalton, \emph{Nonlinear commutators in interpolation theory},
 Mem. Amer. Math. Soc. 385 (1988).

\bibitem{kaltdiff} N.J. Kalton,
\emph{Differentials of complex interpolation processes for K\"othe function spaces},
Trans. Amer. Math. Soc. 333 (1992) 479--529.

\bibitem{kaltlocb} N.J. Kalton,
\emph{The three space problem for locally bounded F-spaces},
Compositio Mathematica  37 (1978) 243--276.


\bibitem{kal-mon} N. J. Kalton, S Montgomery-Smith.
\emph{ Interpolation of Banach spaces,} in Handbook of the Geometry of Banach spaces vol. 2, W. B.  Johnson and J. Lindenstrauss (eds.) pp. 1131--1175. North-Holland, Amsterdam 2003.

\bibitem{kaltpeck} N. J. Kalton, N. T. Peck.
\emph {Twisted sums of sequence spaces and the three space problem,} Trans. Amer. Math. Soc. 255  (1979) 1--30.

\bibitem{kato} T. Kato.
\emph{ Perturbation theory for nullity, deficiency and other quantities of linear operators,} J. Analyse Math. 8 (1958) 261--322.

%\bibitem{K} P. Koszmider. \emph {Banach spaces of continuous functions with few operators,}
%Math. Ann. 330 (2004) 151--183.


\bibitem{lindrose} J. Lindenstrauss and H.P. Rosenthal, \emph{Automorphisms
in $c_0, \ell_1$ and $m$}, Israel J. Math.  9 (1969) 227-239.


\bibitem{lindtzaf} J. Lindenstrauss, L Tzafriri.
\emph {Classical Banach Spaces I, Sequence spaces,} Springer-Verlag (1977).

\bibitem{M} V. D. Milman.
\emph { Some properties of strictly singular operators,} Funktsional. Anal. i Prilozhen. 3, No 1 (1969) pp. 93--94 (in Russian). English translation: Funct. Anal. Appl. 3, No 1 (1969) pp. 77--78.

\bibitem{M2} V. D. Milman.
\emph{ Spectrum of bounded continuous functions specified on a unit sphere in Banach space,}
Funktsional. Anal. i Prilozhen. 3, No 2 (1969) pp. 67--79 (in Russian). English translation: Funct. Anal. Appl. 3, No 2 (1969) pp. 137--146.


\bibitem{moretesis} Y. Moreno Salguero, \emph{Theory of $z$-linear maps},
Ph.D. Thesis, Univ. Extremadura, 2003.


\bibitem{moreplic} Y.~Moreno and A.~Plichko, \emph{On automorphic Banach spaces},
Israel J. Math. 169 (2009) 29--45.

\bibitem{pelc} A. Pe{\l}czy{\'n}ski,
\emph {On strictly singular and strictly cosingular operators. I. Strictly singular and strictly cosingular operators
in $C(S)$-spaces,} Bull. Acad. Polon. Sci. S\'er. Sci. Math. Astronom. Phys. 13 (1965) 31--36.

%\bibitem{P} G. Plebanek.
%\emph {A construction of a Banach space $C(K)$ with few operators,} Topology and its Applications 143 (2004) 217--239.

\bibitem{pliv} A. Plichko, \emph{Superstrictly singular and
superstrictly cosingular operators}, in Functional analysis and
its applications, 239--255, North-Holland Math. Stud., 197,
Elsevier, Amsterdam, 2004.


\bibitem{rochberg} R. Rochberg, \emph{Higher order estimates in complex interpolation
theory}, Pacific J. Math. 174 (1996) 247--267.

\bibitem{rochberg-weiss} R. Rochberg and G. Weiss,
\emph{Derivatives of analytic families of Banach spaces,} Ann. of Math. 118 (1983) 315--347.



\bibitem{sua} J. Su\'arez de la Fuente, \emph{A remark about Schatten classes}, Rocky Mtn. J. 44 (2014) 2093-2102.


\bibitem{S2} S. Szarek. \emph {A superreflexive Banach space which does not admit complex structure,} Proc. Amer. Math. Soc. 97 (1986) 437--444.


\end{thebibliography}
\end{document}